\documentclass[11pt,english]{article}
 \usepackage[english]{babel}
\usepackage{cite} 

\usepackage{graphicx,subfigure}

\usepackage[usenames, dvipsnames]{xcolor}
\usepackage[utf8x]{inputenc}
\usepackage[T1]{fontenc}
\usepackage{tikz}
\usepackage{pgfplots}
\usepackage{amsmath}
\usepackage{amsfonts}
\usepackage{amssymb}
\usepackage{amsthm}

\usepackage{esint}

\usepackage[colorlinks=true,urlcolor=blue,
citecolor=red,linkcolor=blue,linktocpage,pdfpagelabels,
bookmarksnumbered,bookmarksopen,backref=page]{hyperref}

 \newcommand{\IM}{\mathrm{Im}\,}
 \newcommand{\RE}{\mathrm{Re}\,}
 
\newtheorem{theorem}{Theorem}[section]

 \newtheorem{proposition}[theorem]{Proposition}
 \theoremstyle{definition}
 \newtheorem{definition}[theorem]{Definition}
 \theoremstyle{remark}
 \newtheorem{remark}[theorem]{Remark}
 
 \numberwithin{equation}{section}

\newcommand{\R}{\mathbb R}
\begin{author}
{Jaime Angulo Pava $^1$ and Ram\'on G. Plaza $^2$}
\end{author}

\begin{title}{Unstable kink-soliton profiles for the sine-Gordon equation on a $\mathcal Y$-junction graph with $\delta$-interaction}
\end{title}
\begin{document}
\maketitle

\centerline{$^1$ Department of Mathematics,
IME-USP}
 \centerline{Rua do Mat\~ao 1010, Cidade Universit\'aria, CEP 05508-090,
 S\~ao Paulo, SP (Brazil)}
 \centerline{\tt angulo@ime.usp.br}
 \centerline{ $^2$ Instituto de Investigaciones en Matem\'aticas Aplicadas
   y en Sistemas,}
 \centerline{Universidad Nacional Aut\'{o}noma de M\'{e}xico,  Circuito Escolar s/n,}
 \centerline{Ciudad Universitaria, C.P. 04510 Cd. de M\'{e}xico (Mexico)}
 \centerline{\tt  plaza@mym.iimas.unam.mx}

\begin{abstract}
The aim of this work is to establish a linear instability result of stationary, kink and kink/anti-kink soliton profile solutions for the sine-Gordon equation on a metric graph with a structure represented by a $\mathcal Y$-junction. The model considers boundary conditions at the graph-vertex of $\delta$-interaction type, or in other words, continuity of the wave functions at the vertex plus a law of Kirchhoff-type for the flux. It is shown that kink and kink/anti-kink soliton type stationary profiles are linearly (and nonlinearly) unstable. For that purpose, a linear instability criterion that provides the sufficient conditions on the linearized operator around the wave to have a pair of real positive/negative eigenvalues, is established. As a result, the linear stability analysis depends upon of the spectral study of this linear operator and of its Morse index. The extension theory of symmetric operators, Sturm-Liouville oscillation results and analytic perturbation theory of operators are fundamental ingredients in the stability analysis. A comprehensive study of the local well-posedness of the sine-Gordon model in $\mathcal E(\mathcal Y) \times L^2(\mathcal{Y})$ where $\mathcal E(\mathcal Y) \subset H^1(\mathcal{Y})$ is an appropriate energy space, is also established. The theory developed in this investigation has prospects for the study of the instability of stationary wave solutions of other nonlinear evolution equations on metric graphs.
\end{abstract}


\textbf{Key  words}.  sine-Gordon model, metric graphs, kink and anti-kink solutions, $\delta$-type interaction, analytic perturbation theory, extension theory, instability.

\section{Introduction}

In recent years, there has been a growing interest among the scientific community in modeling and analyzing evolution problems described by partial differential equations (PDEs) on \emph{graphs}. A metric graph is a network-shaped structure of edges which are assigned a length (that is, a metric) and connected at vertices according to boundary conditions which determine the dynamics on the network. This trend has been mainly motivated by the demand of reliable mathematical models for different phenomena in branched systems which, in meso- or nano-scales, resemble a thin neighborhood of a graph, such as Josephson junction networks \cite{NakO76,NakO78}, electric circuits \cite{BaCh13}, unidirectional shallow water flow in a network \cite{BoCa08}, blood pressure waves in large arteries \cite{McDo11}, or nerve impulses in complex arrays of neurons \cite{Sco03}, just to mention a few examples (see also \cite{Berkolaiko, BK, BeK, BurCas01, CaoMa95, Fid15, Kuch04,  Mehmeti, Mug15} and the many references therein). One of the main difficulties consists on the fact that metric graphs are not manifolds. From a mathematical viewpoint, the nature of a PDE model on a graph is tantamount to a system of PDEs defined on appropriate intervals in which the coupling is given exclusively through the boundary conditions at the vertices, known as the ``topology of the graph'' (see, for example, a recent review of the extension of Hamiltonian dynamics to non-manifold structures by Bibikov and Prokhorov \cite{BiP09}). Hence, both the model equation and the geometry are complex in general, making the problem difficult to tackle. A first step is to consider simple geometries, such as star-shaped graphs and $\mathcal{Y}$-junctions. Another simplification is to solve linear equations, such as Schr\"{o}dinger operators, on graphs. In this case the system is called a \emph{quantum graph} and there is a broad literature on the subject (see, e.g., \cite{Berko17,Berkolaiko,BK,BlaExn08,Kuch04,Kuch05}).

The extension of the analysis to \emph{nonlinear} dispersive equations on graphs is an emerging subfield that has recently attracted the attention of mathematicians and physicists alike. In particular, the prototype of graph geometry often considered is the \emph{star graph}, namely, a metric graph with $N$ semi-infinite edges of the form $(0,\infty)$ connected at a single common vertex at $\nu = 0$. The analyses have focused on the characterization of ground states and standing waves. These works pertain primarily to the nonlinear Schr\"{o}dinger (NLS) equation (see Adami \emph{et al.} \cite{AdaNoj12a,AdaNoj14,AdaNoj16}, Angulo and Goloschapova \cite{AngGol18b,AngGol18a} and Cacciapuoti \emph{et al.} \cite{CFN17}; see also Noja \cite{Noj14} for a recent review), albeit other nonlinear dispersive equations have been also studied, such as the Benjamin-Bona-Mahony (BBM) equation for unidirectional shallow fluid flow on a $\mathcal{Y}$-junction (see Bona and Cascaval \cite{BoCa08} and Mugnolo and Rault \cite{Mugnolobbm}), Airy-type equations (Mugnolo \emph{et al.} \cite{MNS}), nonlinear Klein-Gordon equations (Goloschapova \cite{Gol20}) or the Korteweg-de Vries (KdV) equation on general metric graphs (Angulo and Cavalcante \cite{AngCav}). All these model equations share one feature: the presence of solitary wave solutions (solitons). The analysis of existence, stability and the overall role of solitons for some PDE models on graphs, as well as the study of nonlinear equations on ramified structures, constitute a very active field of research due to its potential of becoming a paradigm model for topological effects of nonlinear wave propagation. The objective of this work is to contribute to this on-going effort through the analysis of the well-known sine-Gordon equation on a metric graph of $\mathcal{Y}$-junction type.

\subsection{The sine-Gordon equation on graphs}

The sine-Gordon equation in one space dimension,
\begin{equation}
\label{sine-G}
u_{tt} - c^2 u_{xx} + \sin u = 0,
\end{equation}
where $c > 0$ is a constant and $x \in \R$, $t > 0$, appears in many models in mathematical physics, such as the description of the magnetic flux in a Josephson line \cite{BEMS,BaPa82,SCR}, crystal dislocations \cite{FreKo}, mechanical oscillations of a nonlinear pendulum \cite{Dra83}, or even nonlinear oscillations in DNA chains \cite{IvIv13}, among other applications. It is a nonlinear wave equation underlying many important mathematical features, such as complete integrability \cite{AKNS73,AKNS74}, a Hamiltonian structure \cite{TaFa76} and the existence of localized solutions (solitons) \cite{SCM,Sco03}.

Posing the sine-Gordon equation on a metric graph comes out naturally from practical applications. For example, since the phase-difference in a long (infinite) Josephson junction obeys equation \eqref{sine-G}, the coupling of two or more Josephson junctions forming a network can be effectively modeled by the sine-Gordon model on a graph. The sine-Gordon equation was first conceived on a $\mathcal{Y}$-shaped Josephson junction by Nakajima \emph{et al.} \cite{NakO76,NakO78} as a prototype for logic circuits. The authors consider three long (semi-infinite) Josephson junctions coupled at one single common vertex, a structure known as a \emph{tricrystal junction}. There exists two main types of $\mathcal{Y}$-junctions. A $\mathcal{Y}$-junction of the first type (or type I) consists of one incoming (or parent) edge, $E_1 = (-\infty,0)$, meeting at one single vertex at the origin, $\nu = 0$, with other two outgoing (children) edges, $E_j = (0,\infty)$, $j = 2,3$. The second type (or $\mathcal{Y}$-junction of type II) resembles more a starred structure and consists of three identical edges of the form $E_j = (0,\infty)$, $1 \leqq j \leqq 3$. See Figure \ref{figYjunction} for an illustration. Junctions of type I are more common in unidirectional fluid flow models (see, for example, \cite{BoCa08}), whereas graphs of type I or II are indistinctively used to describe Josephson tricrystal junctions; see, for instance, \cite{Grunn93,Susa19} (type I), or \cite{SvG05,KCK00,Sabi18} (type II). In the present case of the sine-Gordon equation \eqref{sine-G}, the choice of a junction of either type makes no difference in the stability analysis.

\begin{figure}[t]
\begin{center}
\subfigure[$\mathcal{Y} = (-\infty,0) \cup (0,\infty) \cup (0,\infty)$]{\label{figYtipoI}\includegraphics[scale=.4, clip=true]{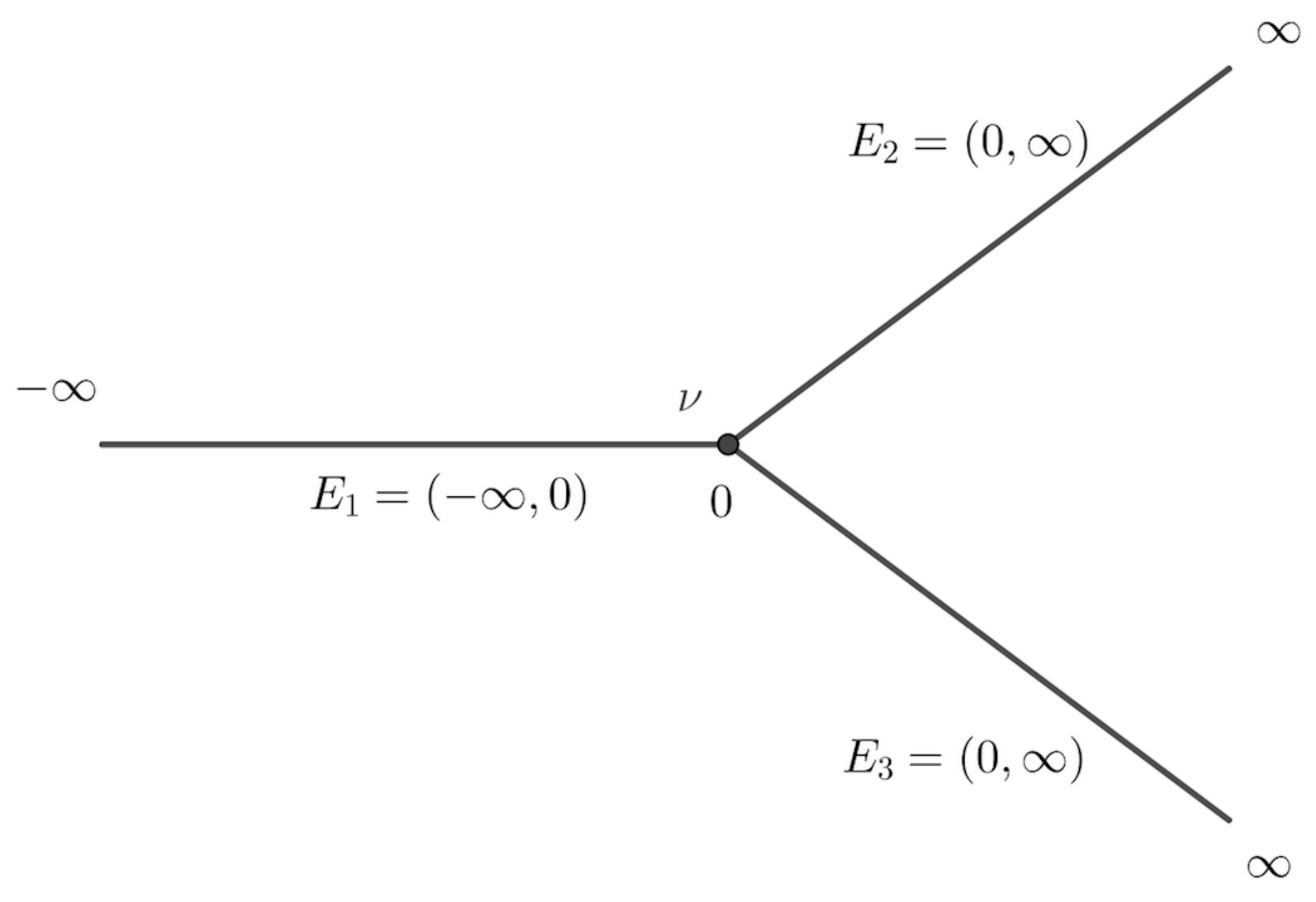}}
\subfigure[$\mathcal{Y} = (0,\infty) \cup (0,\infty) \cup (0,\infty)$]{\label{figYtipoII}\includegraphics[scale=.4, clip=true]{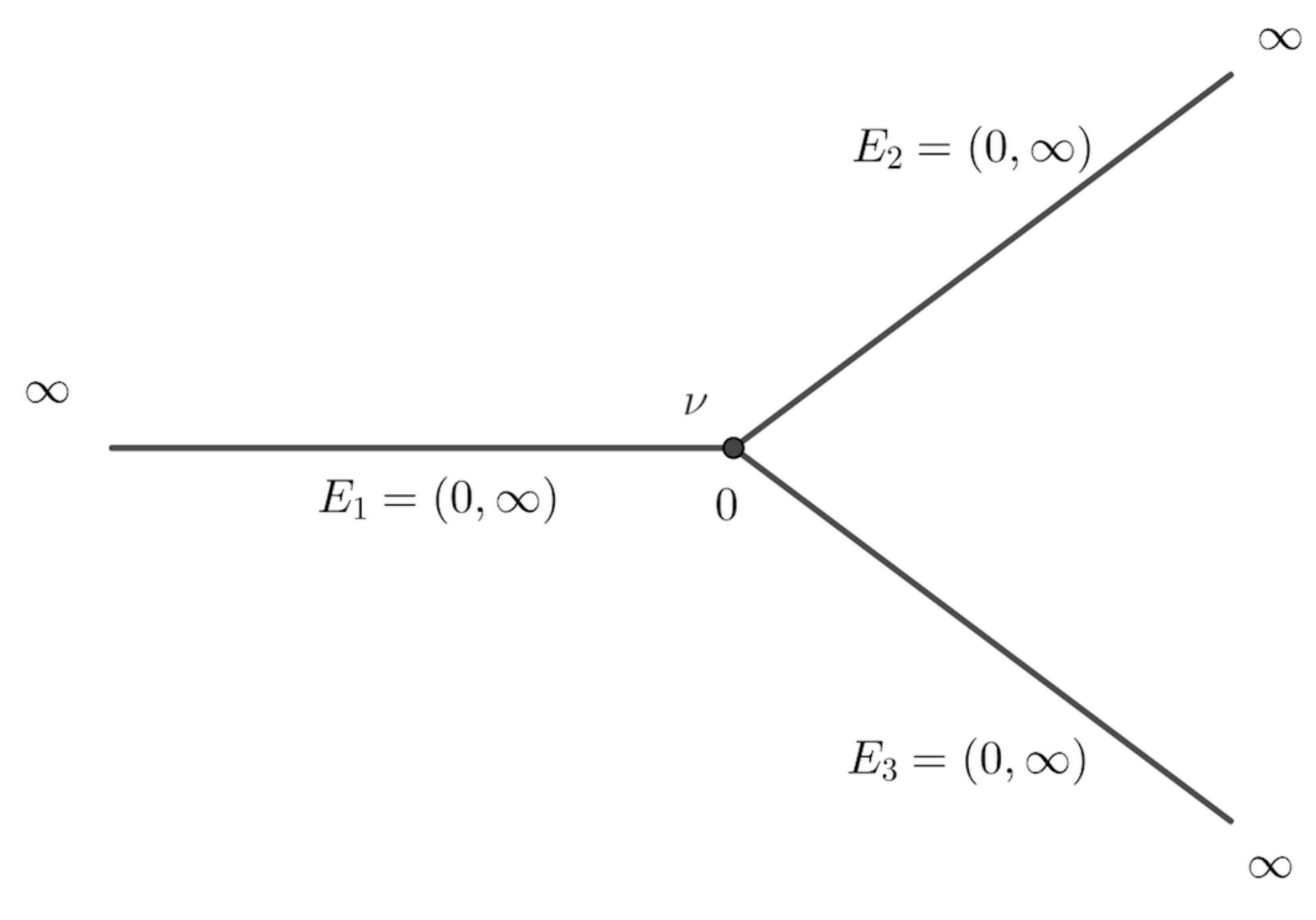}}
\end{center}
\caption{\small{Panel (a) shows a $\mathcal{Y}$-junction of the first type with $E_1 = (-\infty,0)$ and $E_j = (0,\infty)$, $j=2,3$, whereas panel (b) shows a $\mathcal{Y}$-junction of the second type (star graph) with $E_j = (0,\infty)$, $1 \leqq j \leqq 3$.}}\label{figYjunction}
\end{figure}

What is more crucial is the choice of boundary conditions, mainly because the transition rules at the vertex completely determine the dynamics of the PDE model on the graph. For the sine-Gordon equation in $\mathcal{Y}$-junctions, previous studies have basically (and almost exclusively) considered two types of boundary conditions: interactions of $\delta$-type, and of $\delta'$-type. The former refers to continuity of the wave functions and a balance flux relation for the derivatives of the wave functions at the vertex. The latter consists of continuity of the fluxes (derivatives) at the vertex (surface current density is the same in all three thin films at the intersection), and a Kirchhoff-type rule for the \emph{self-induced} magnetic flux. Since Josephson models arise in the description of electromagnetic flux, interactions of $\delta'$-type have received more attention (see, for example, \cite{Grunn93,KCK00,SvG05} for the description and analysis of stationary kink-type solutions and, more recently, \cite{Susa19} for solutions of the \emph{breather} type). In both cases ($\delta$- and $\delta'$-types), rigorous studies of the well-posedness of the model, as well as of the spectral and nonlinear stability properties of particular stationary solutions, are still under development (the case of interactions of $\delta'$-type at the vertex will be addressed in a companion paper \cite{AnPl}).

\subsection{Boundary conditions of $\delta$-interaction type on a $\mathcal Y$-junction}

In this paper, we focus our attention to boundary conditions of $\delta$-type. Previous works (see, e.g., \cite{Caput14,DuCa18,Sabi18}) have considered interactions consisting of two basic transition rules at the vertex. For concreteness, we describe them in the context of a $\mathcal{Y}$-junction of type I. The first boundary condition is the wave function continuity at the intersection point, namely,
\begin{equation}
\label{contbc}
u_1(0-) = u_2(0+) = u_3(0+).
\end{equation}
In the context of $\mathcal{Y}$-junctions for the sine-Gordon model, it was first proposed by Nakajima \emph{et al.} \cite{NakO76} (see equation (4) in that reference) to account for circuits with a \emph{trigger turning point}.
The second boundary condition reads,
\begin{equation}
\label{Kirchhoffbc}
-c_1^2 \partial_x u_1(0-) + \sum_{j=2}^3 c_j^2 \partial_x u_j(0+) = 0,
\end{equation}
and it is equivalent to the charge conservation property or conservation of the \emph{current flow} (the Kirchhoff law for electric currents in the case of a Josephson junction, for example) at the vertex. This boundary condition, adopted from previous studies for (linear) Klein-Gordon equations (see, e.g., \cite{BiP09,MeRe03}), is a continuous analogue of the celebrated Kirchhoff's circuit law in the sense that it somehow expresses a flux balance across the vertex. Dutykh and Caputo \cite{DuCa18} have shown that \eqref{Kirchhoffbc} can be justified by transforming the $\mathcal{Y}$-junction domain into a manifold $\mathcal{Y}_\varepsilon$ of small thickness $\varepsilon > 0$. In this fashion, $\mathcal{Y}_\varepsilon$ becomes a tubular neighborhood of the graph, often referred to in the literature as a \emph{fat graph} (see \cite{BlaExn08}, chapter 17). Henceforth, imposing Neumann boundary conditions on $\partial \mathcal{Y}_\varepsilon$ and taking the limit when $\varepsilon \to 0$ leads to \eqref{Kirchhoffbc}. Another approach adopted by the same authors in order to derive \eqref{Kirchhoffbc} is based on a conservation of energy argument (see \cite{DuCa18} for further information).

The sine-Gordon model on a $\mathcal{Y}$-junction with this type of boundary conditions of $\delta$-interaction type has been studied only by Caputo and Dutykh \cite{Caput14,DuCa18} and by Sabirov \emph{et al.} \cite{Sabi18}, up to our knowledge. In the latter reference, the authors consider the stationary sine-Gordon equation on a $\mathcal{Y}$-junction with \emph{finite} edges, $E_j = (0,L_j)$. They find exact analytical solutions under boundary conditions of both the $\delta$- and $\delta'$-interaction types. Caputo and Dutykh \cite{Caput14,DuCa18} formulate the sine-Gordon model on a $\mathcal{Y}$-junction under boundary conditions \eqref{contbc} and \eqref{Kirchhoffbc}, and implement a symplectic numerical scheme to solve it and, more precisely, to numerically study soliton collisions at the vertex. 

As far as we know, there is no analytical study of the stability of stationary solutions to the sine-Gordon model on a graph with boundary conditions of $\delta$-interaction type available in the literature. The stability of these static configurations is an important property from both the mathematical and the physical points of view. Stability can predict whether a particular state can be observed in experiments or not. Unstable configurations are rapidly dominated by dispersion, drift, or by other interactions depending on the dynamics, and they are practically undetectable in applications. In the analysis of their stability, it is customary to \emph{linearize} the equation around the profile solution and to obtain useful information from the spectral properties of the linearized operator posed on an appropriate function space. Upon linearization of the sine-Gordon equation \eqref{sine-G} around a stationary soliton solution, we end up with a Schr\"odinger type operator with a bounded potential (see the form of the operator \eqref{trav23} below) that can be appropriately defined on a graph. Therefore, motivated by the spectral analysis of the linearized model around the static profile solutions, we adopt a quantum-graph approach in order to justify, interpret and extend the boundary conditions that actually define the model.

According to custom in quantum graph theory \cite{Berko17,BlaExn08}, let us consider for simplicity the case of a star graph $\mathcal{G}$ constituted by $N$ semi-infinite edges of the form $E_j = (0,\infty)$, $1 \leqq j \leqq N$, attached at a single vertex at $\nu = 0$. A function on $\mathcal{G}$ is a vector $\mathbf{u} = (u_j)_{j=1}^N$, with scalar components, $u_j = u_j(x)$ on each edge $E_j$. Sobolev and Lebesgue spaces on $\mathcal{G}$ are defined as $H^m(\mathcal{G}) = \oplus_{j=1}^N H^m(0,\infty)$ and $L^p(\mathcal{G}) = \oplus_{j=1}^N L^p(0,\infty)$, respectively. Schr\"odinger operators on quantum graphs have the form
\[
\widetilde{\mathcal{L}} \mathbf{u} = \left\{ \Big( - \frac{d^2}{dx^2} + V_j(x) \Big) u_j \right\}_{j=1}^N,
\]
defined on $L^2(\mathcal{G})$ with a domain being a subset of $H^2(\mathcal{G})$. If the potentials $V_j$ are not too singular then the coupling at the vertex does not depend on them and the self-adjoint extensions of the Laplace operator determine all the self-adjoint extensions of $\widetilde{\mathcal{L}}$. It is known (see, e.g., \cite{BlaExn08}) that all self-adjoint extensions of the formal operator $-\Delta = \{ - u_j'' \}_{j=1}^N$ on the star graph are determined by vertex conditions having the form 
\[
(U-I) \mathbf{u}(0) + i(U+I) \mathbf{u}'(0) = 0,
\]
where $U$ is a unitary matrix. A self-adjoint extension is of $\delta$-\emph{interaction type} in the particular case where the matrix $U$ is given by $U_{jk} = 2(N+iZ)^{-1} - \delta_{j,k}$, for $1 \leqq j,k \leqq N$, being $\delta_{j,k}$ the Kronecker symbol and for an arbitrary parameter $Z \in \R$. Upon substitution into the last equation we obtain the so called \emph{$\delta$-boundary conditions at the vertex with intensity $Z$},
\[
\begin{aligned}
u_1(0) = u_2(0) = \ldots &= u_N(0), \\
\sum_{j=1}^n u'_j(0) &= Z u_1(0).
\end{aligned}
\]
In such a case, the self-adjoint extension is defined as the formal operator, $-\Delta_Z \equiv -\Delta$, on a domain, $D(-\Delta_Z)$, which is a subspace of $H^2(\mathcal{G})$ that includes the $\delta$-conditions at the vertex. The parameter value $Z$ is fundamental an determines the basic spectral properties of the operator. For example, it can be shown that $-\Delta_Z$ has non-empty point spectrum only when $Z < 0$, yielding the term ``attractive'' to characterize the $\delta$-vertex with $Z < 0$, in contrast with a ``repulsive'' vertex when $Z \geq 0$. The former can be interpreted as an attractive potential well at the vertex. When $Z = 0$ the $\delta$-condition at the vertex is said to be of Kirchhoff type. The quadratic form associated to $-\Delta_Z$ is
\[
Q[\mathbf{u}] = \frac{1}{2} \| \mathbf{u}' \|^2_{L^2(\mathcal{G})} + \frac{Z}{2} |u_1(0)|^2,
\]
with domain
\[
D(Q) = \left\{ \mathbf{u} \in H^1(\mathcal{G}) \, : \, u_1(0) = \ldots =u_N(0) \right\} =: \mathcal{E}(\mathcal{G}),
\]
independent of $Z$ and usually referred to as the \emph{energy domain} (see \cite{BlaExn08,Noj14}).

For the sine-Gordon model on a graph, keeping the characteristic velocity on each edge is important. Thus, we are concerned with all the self-adjoint extensions of the formal operator 
\[
\mathcal{F} \mathbf{u} = \left\{ \Big( -c_j^2 \frac{d^2}{dx^2} \Big) u_j \right\}_{j=1}^N, \qquad \mathbf{u} = (u_j)_{j=1}^N,
\]
on a star graph $\mathcal{G}$. It is not hard to verify that self-adjoint extensions correspond to an interaction of $\delta$-type only when $U_{jk} = 2c_k^2(N+iZ)^{-1} - \delta_{j,k}$, even though in this case the matrix $U$ is no longer unitary. For convenience of the reader, we provide a direct proof of this fact in Appendix \S \ref{secApp} for the particular case of a $\mathcal{Y}$-junction of type I, whereupon substitution of $U$ yields the transition conditions 
\begin{equation}
\label{bcI}
\begin{aligned}
u_1(0-) = u_2(0+) &= u_3(0+),\\
-c_1^2  u'_1(0-) + \sum_{j=2}^3 c_j^2  u'_j(0+) &= Z u_1(0-),
\end{aligned}
\end{equation}
recovering in this fashion the continuity condition \eqref{contbc} and the Kirchhoff boundary condition \eqref{Kirchhoffbc} when $Z = 0$.

Therefore, the main goal of this paper is to analyze the structural and stability properties of stationary solutions to the sine-Gordon model defined on a $\mathcal{Y}$-junction under boundary conditions of $\delta$-interaction type of the form \eqref{bcI} at the vertex. These conditions depend upon the parameter $Z$, which ranges along the whole real line and determines the dynamics of the solutions. Therefore, the value $Z \in \R$ is part of the physical parameters that define the physical model (such as the speeds $c_j$, for instance). Instead of adopting \emph{ad hoc} boundary conditions, we consider a parametrized family of transition rules covering a wide range of applications and which, for the particular value $Z = 0$, include the Kirchhoff condition \eqref{Kirchhoffbc} previously studied in the literature. Our analysis focuses on a particular class of solutions of the sine-Gordon equation known as \emph{kinks} (also referred to as topological solitons \cite{Dra83,Sco03,SCM}). For completeness, we also include the stability analysis for static configurations of kink/anti-kink type.

\subsection{Main results}

In this paper we consider the sine-Gordon equation \eqref{sine-G} on a metric graph with the shape of a $\mathcal{Y}$-junction with three semi-infinite edges and joined by a single vertex $\nu = 0$. For concreteness, in the sequel we assume that the $\mathcal{Y}$-junction is of type I, where $E_1 = (-\infty,0)$ and $E_j = (0,\infty)$, $j = 2,3$. The results and observations of this paper can be easily extended to the case of a $\mathcal{Y}$-junction of type II at the expense of extra bookkeeping. The sine-Gordon model on the $\mathcal{Y}$-junction under consideration reads
\begin{equation}
\label{sg1}
\partial_t^2 u_j - c_j^2 \partial_x^2 u_j + \sin u_j= 0, \qquad x \in E_j, \;\,  t > 0, \;\, 1 \leqq j \leqq 3,
\end{equation}
where $\mathbf{u} = (u)_{j=1}^3$, $u_j = u_j(x,t)$. It is assumed that the characteristic speed on each edge $E_j$ is constant and positive, $c_j > 0$, without loss of generality. Clearly, one can recast the equations in \eqref{sg1} as a first order system that reads
\begin{equation}
\label{sg2}
\begin{cases}
\partial_t u_j = v_j\\
\partial_t v_j =c_j^2 \partial_x^2 u_j - \sin u_j,
\end{cases}
\qquad x \in E_j, \;\,  t > 0, \;\, 1 \leqq j \leqq 3.
\end{equation}
The equations are endowed with boundary conditions of $\delta$-type, having the the form \eqref{bcI} for all $t > 0$ and for a given parameter $Z \in \R$. 

We are interested in the dynamics generated by the flow of the sine-Gordon model \eqref{sg1} around solutions of stationary type,
\begin{equation}\label{trav20}
u_j (x,t) = \phi_j(x), \qquad v_j(x,t) = 0,
\end{equation}
for all $j = 1,2,3$, and $x \in E_j$, $t > 0$, where each of the profile functions $\phi_j$ satisfies the equation
\begin{equation}
\label{trav21}
-c_j^2 \phi''_j + \sin \phi_j=0, 
\end{equation}
on each edge $E_j$ and for all $j$, as well as the boundary conditions \eqref{bcI} at the vertex $\nu = 0$. More precisely, we consider the particular family of profiles determined by the well-known kink-soliton profile solutions to the sine-Gordon equation on the full real line \cite{Dra83,SCM}, having the form
\begin{equation}
\label{trav22}
\begin{cases}
\phi_1(x) = 4 \arctan \big( e^{(x-a_1)/c_1}\big), & x \in (-\infty,0),\\
\phi_j(x) = 4 \arctan \big( e^{-(x-a_j)/c_j}\big), & x \in (0,\infty), \,\; j=2,3,\\
\end{cases}
\end{equation}
where each $a_j$ is a constant determined by the boundary conditions \eqref{bcI}. Notice as well that this family of stationary solutions \eqref{trav22} satisfies
\begin{equation}
\label{bcinfty}
\phi_1(-\infty) = \phi_j(+\infty) = 0, \qquad j = 2,3
\end{equation}
(in other words, the constant of integration when solving \eqref{trav21} to arrive at \eqref{trav22} is zero on each edge $E_j$). This decaying behavior at $\pm \infty$, for instance, guarantees that $\Phi = (\phi_j)_{j=1}^3 \in H^2(\mathcal{Y})$.

In the forthcoming stability analysis, the family of linearized operators around the stationary profiles plays a fundamental role. These operators are characterized by the following formal self-adjoint diagonal matrix operators,
\begin{equation}\label{trav23}
\mathcal{L} \bold{v}=\Big (\Big(-c_j^2\frac{d^2}{dx^2}v_j + \cos (\phi_j)v_j
\Big)\delta_{j,k} \Big ),\quad \l1\leqq j, k\leqq 3,\;\;\bold{v}= (v_j)_{j=1}^3,
\end{equation}
where $\delta_{j,k}$ denotes the Kronecker symbol, and defined on domains with $\delta$-type interaction at the vertex $\nu = 0$,
\begin{equation}\label{2trav23}
D(\mathcal{L}_Z)= \Big \{\bold{v}=(v_j)_{j=1}^3\in H^2(\mathcal{Y}):
v_1(0-)=v_2(0+)=v_3(0+),\;\; \sum\limits_{j=2}^3 c_j^2v_j'(0+)-c_1^2v_1'(0-)=Zv_1(0-) \Big \},
\end{equation}
with $Z \in \R$. The operator is completely determined, $(\mathcal{L}_Z, D(\mathcal{L}_Z))$, $\mathcal{L}_Z \equiv \mathcal{L}$, by each parameter value $Z \in \R$. It is to be observed that the particular family \eqref{trav22} of kink-profile stationary solutions under consideration is such that $\Phi = (\phi_j)_{j=1}^3 \in D(\mathcal{L}_Z)$.

Motivated by physical considerations, we also study static solutions to \eqref{sg1} on a $\mathcal{Y}$-junction of \emph{anti-kink} type, which represent waves pinned at the vertex and are of interest in the studies of impurities of the medium (modeled, in this case, by the vertex itself; see, e.g., \cite{SCM} for an interpretation of such impurities in the real line). Hence, we also consider one anti-kink on the parent edge, $E_1 = (-\infty,0)$, coupled with two kinks on the remaining edges, $E_j = (0,\infty)$, $j=2,3$. Such static solutions of kink/anti-kink type have the form
\begin{equation}
\label{trav-akink}
\begin{cases}
\phi_1(x) = 4 \arctan \big( e^{-(x-a_1)/c_1}\big), & x \in (-\infty,0),\\
\phi_j(x) = 4 \arctan \big( e^{-(x-a_j)/c_j}\big), & x \in (0,\infty), \,\; j=2,3.\\
\end{cases}
\end{equation}
Notice that, in this case,
\[
\lim_{x\to -\infty}\phi_1(x)=2\pi, \qquad \lim_{x\to +\infty}\phi_j(x)=0, \; \; \; j =2,3,
\]
and therefore the configuration $\Phi = (\phi_j)_{j=1}^3$ \emph{is not in} $H^2(\mathcal{Y})$. Nonetheless, it is possible to linearize the equation around this static solution and to define a suitable linear operator on the same domain in the energy space endowed with the $\delta$-coupling at the vertex. The spectral analysis of this operator can be performed in a similar fashion.

\bigskip

Let us summarize the main contributions of this paper:
\begin{itemize}
\item[$-$] First, we prove that the Cauchy problem associated to \eqref{sg2} is well-posed in the energy space $\mathcal{E}(\mathcal{Y}) \times L^2(\mathcal{Y})$ (section \S \ref{seclocalWP}), where 
\[
\mathcal E(\mathcal Y) = \{(v_j)_{j=1}^3\in H^1(\mathcal{Y}):
v_1(0-)=v_2(0+)=v_3(0+) \}.
\]
This is the content of Theorem \ref{well0}. Even though the well-posedness result is not part of the spectral stability analysis, it is fundamental to reach a nonlinear conclusion; see Remark \ref{remnolineal} below.
\item[$-$] In section \S \ref{seccriterium} we establish a general instability criterion for stationary solutions for the sine-Gordon model \eqref{sg2} on a $\mathcal{Y}$-junction. The reader can find this result in Theorem \ref{crit} below. It essentially provides sufficient conditions on the flow of the semigroup generated by the linearization around the stationary solutions, for the existence of a pair of positive/negative real eigenvalues of the linearized operator based on its Morse index. It is to be observed that this instability criterion applies to any type of stationary solutions (such as kinks or breathers, for example) and for other self-adjoint extensions characterized by different interactions at the vertex, such as the $\delta'$-type (see \cite{AnPl}), making it potentially useful in applications.
\item[$-$] We provide a complete characterization of the stationary kink-profile solutions \eqref{trav22} in terms of the parameter $Z \in \R$. It is shown that the family $\Phi_Z = (\phi_j)_{j=1}^3$ belongs to the domain space $D(\mathcal{L}_Z)$ only for parameter values $Z \in (-\sum_{j=1}^3 c_j, 0)$ (section \S \ref{secprofiles}). This implies, in turn, that there does not exist a kink-profile solution of the form \eqref{trav22} (and hence, satisfying the boundary condition at $\pm \infty$, \eqref{bcinfty}), compatible with the Kirchhoff condition \eqref{Kirchhoffbc} with $Z = 0$.
\item[$-$] We show (see section \S \ref{secspecstudy}) that the family of stationary kink-profiles, $Z \mapsto \Phi_Z$, are spectrally unstable under the flow of the sine-Gordon model for each $Z \in (-\sum_{j=1}^3 c_j, 0)$. This is the content of the main Theorem \ref{1main} below. The proof is divided into different steps (see Propositions \ref{main3} and \ref{main4}), in order to show that the Morse index of the linearization is exactly equal to one for the different parameter values of $Z$ under consideration. As we mentioned above, this result implies the nonlinear (orbital) instability of the kink-profiles (see Remark \ref{remnolineal}).
\item[$-$] Due to their importance in applications and in order to illustrate the range of applicability of the linear instability criterion developed here, we also establish the spectral instability of the kink/anti-kink profiles of the form \eqref{trav-akink} (see section \S \ref{aksect}). In view that the latter do not belong to the energy space, in section \ref{akfunspa} we verify the hypotheses of our linear instability criterion with respect to the flow generated by \emph{finite energy perturbations} of the static solutions (for a similar analysis on the semigroup generated by perturbations of unbounded subluminal rotations for the sine-Gordon model, see \cite{AnPl16}). Whenceforth, the spectral analysis follows similarly as in the previous kink configuration, yielding the spectral instability result for the kink/anti-kink profiles as well (see Theorem \ref{akmain}).
\item[$-$] For completeness, in Appendix \S \ref{secApp} we prove that all self-adjoint extensions of the formal operator \eqref{trav23} on a $\mathcal{Y}$-junction are defined on domains of the form \eqref{2trav23}.
\end{itemize}

\subsection*{On notation}

Let $A$ be a  closed densely defined symmetric operator in a Hilbert space $H$. The domain of $A$ is denoted by $D(A)$. The deficiency indices of $A$ are denoted by  $n_\pm(A):=\dim  \ker (A^*\mp iI)$, with $A^*$ denoting the adjoint operator of $A$.  The number of negative eigenvalues counting multiplicities (or Morse index) of $A$ is denoted by  $n(A)$. For any $-\infty\leq a<b\leq\infty$, we denote by $L^2(a,b)$ the Hilbert space equipped with the inner product 
\[
(u,v)=\int_a^b u(x)\overline{v(x)}dx. 
\]
By $H^n(a,b)$  we denote the classical  Sobolev spaces on $(a,b)\subseteq \mathbb R$ with the usual norm.   We denote by $\mathcal{Y}$ the junction parametrized by the edges $E_1 = (-\infty,0)$, $E_j = (0,\infty)$, $j =2,3$,  attached to a common vertex $\nu=0$. On the graph $\mathcal{Y}$ we define the classical spaces 
  \begin{equation*}
  L^p(\mathcal{Y})=L^p(-\infty, 0)  \oplus L^p(0, +\infty) \oplus L^p(0, +\infty), \quad \,p>1,
  \end{equation*}   
 and 
  \begin{equation*}  
 \quad H^m(\mathcal{Y})=H^m(-\infty, 0) \oplus H^m(0, +\infty)  \oplus H^m(0, +\infty), 
 \end{equation*}   
with the natural norms. Also, for $\mathbf{u}= (u_j)_{j=1}^3$, $\mathbf{v}= (v_j)_{j=1}^3 \in L^2(\mathcal{Y})$, the inner product is defined by
$$
\langle \mathbf{u}, \mathbf{v}\rangle= \int_{-\infty}^0 u_1(x) \overline{v_1(x)} \, dx  + \sum_{j=2}^3 \int_0^{\infty}  u_j(x) \overline{v_j(x)} \, dx
$$
Depending on the context we will use the following notations for different objects. By $\|\cdot \|$ we denote  the norm in $L^2(\mathbb{R})$ or in $L^2(\mathcal{Y})$. By $\| \cdot\| _p$ we denote  the norm in $L^p(\mathbb{R})$ or in $L^p(\mathcal{Y})$.  For the case of $\mathcal{Y}$ being a junction of type II with $\mathcal{Y}=(0, +\infty)\cup(0, +\infty)\cup(0, +\infty)$, similar definitions as above can be  given.

 \section{Local well-posedness theory for the sine-Gordon model in $\mathcal E(\mathcal Y)\times L^2(\mathcal{Y})$}
\label{seclocalWP}

In this section we study the local well-posedness problem associated to \eqref{sg2} initially with a specific framework.  We recast system \eqref{sg2} in the vectorial form 
  \begin{equation}\label{stat1}
  \bold w_t=JE\bold w +F(\bold w)
 \end{equation} 
  where $\bold w=(u, v)^\top$, with $u=(u_1, u_2, u_3)$, $v=(v_1, v_2, v_3)$, $u_1, v_1: (-\infty, 0)\to \mathbb R$, $u_j, v_j: (0, +\infty)\to \mathbb R$, $j=2,3$,
 \begin{equation}\label{stat2} 
J=\left(\begin{array}{cc} 0 & I_3 \\ -I_3  & 0 \end{array}\right),\quad E=\left(\begin{array}{cc} \mathcal F& 0 \\0 & I_3\end{array}\right), \quad 
F(\bold w)=\left(\begin{array}{cc}  0\\  0   \\  0   \\  -\sin (u_1)   \\  -\sin (u_2)   \\  -\sin (u_3)  \end{array}\right)
\end{equation} 
 where $I_3$ denotes the identity matrix of order $3$ and $ \mathcal F$ the diagonal-matrix linear operator
 \begin{equation*} 
 \mathcal{F}=\Big (\Big(-c_j^2\frac{d^2}{dx^2}
\Big)\delta_{j,k} \Big ),\quad\l1\leqq j, k\leqq 3.
   \end{equation*}

Here we will consider the operator $\mathcal F_Z\equiv\mathcal F$ defined on the $\delta$-interaction domain $D(\mathcal F_Z)$ 
\begin{equation}\label{domainZ}
D(\mathcal{F}_Z)= \Big \{\bold{v}=(v_j)_{j=1}^3\in H^2(\mathcal{Y}):
v_1(0-)=v_2(0+)=v_3(0+),\;\; \sum\limits_{j=2}^3 c_j^2v_j'(0+)-c_1^2v_1'(0-)=Zv_1(0-) \Big \}.
\end{equation}
Thus, the natural space to looking for a local well-posedness theory for \eqref{stat1} will be the space $\mathcal E(\mathcal Y)\times L^2(\mathcal Y)$ where $\mathcal E(\mathcal Y)$ represents the closed (continuous) subspace at zero of $H^1(\mathcal{Y})$,
\begin{equation}\label {E}
\mathcal E(\mathcal Y)=\{(v_j)_{j=1}^3\in H^1(\mathcal{Y}):
v_1(0-)=v_2(0+)=v_3(0+) \}.
\end{equation}

The analysis  of the initial value problem for the sine-Gordon vectorial model \eqref{stat1} on metric star shaped graphs requires new tools to those usually used in the case of the model on spaces of $\mathbb R^n$-type.

We start our analysis by establishing the spectrum properties of the family of self-adjoint operator $(\mathcal F_Z, D(\mathcal F_Z))$. 

\begin{theorem}\label{spectrum}
Let $Z\in \mathbb R$. Then  the essential spectrum of $(\mathcal F_Z, D(\mathcal F_Z))$ is purely absolutely continuous and $\sigma_{\mathrm{ess}}(\mathcal F_Z)=\sigma_{\mathrm{ac}}(\mathcal F_Z)=[0,+\infty)$. If $ Z<0$, $\mathcal F_Z$ has precisely one negative, simple eigenvalue, {\it i.e.} its point spectrum $\sigma_{\mathrm{pt}}(\mathcal F_Z)$ is 
$$
\sigma_{\mathrm{pt}}(\mathcal F_Z)=\Big \{-\frac{Z^2}{(\sum_{j=1}^3 c_j )^2}\Big\},
$$ 
with $\Phi_Z=(e^{\alpha x}, e^{-\alpha x}, e^{-\alpha x})$ its  ``strictly positive'' eigenfunction and $\alpha=-Z/\sum_{j=1}^3 c_j >0 $. If $ Z\geqq 0$, $\mathcal F_Z$ has no eigenvalues, $\sigma_{\mathrm{pt}}(\mathcal F_Z)= \varnothing $.
\end{theorem}

\begin{proof} By convenience of the reader, we present the main steps of the proof:
\begin{enumerate}
\item[1)] For every $Z$, the Morse index of $\mathcal F_Z$, $n(\mathcal F_Z)$,  satisfies $n(\mathcal F_Z)\leqq 1$: Indeed, from Proposition \ref{M} we have  immediately that the symmetric operator $(\mathcal M, D(\mathcal M))$ defined in \eqref{M1} (see Appendix \ref{secApp}) is non-negative, $\langle \mathcal M \bold{v},\bold v\rangle\geqq 0$ for every $\bold v \in D(\mathcal M)$, and the deficiency indices are $n_{\pm}(\mathcal M)=1$. Therefore, from Proposition \ref{semibounded} follows that $n(\mathcal F_Z)\leqq 1$. 

\item[2)] For $Z>0$, $n(\mathcal F_Z)=0$: Indeed,  for every $ \bold{v}=(v_j)_{j=1}^3\in D(\mathcal F_Z)$ we have 
$$
\langle \mathcal F_Z \bold{v},\bold v\rangle=\int_{-\infty}^0 c_1^2(v_1)^2 dx +\sum_{j=2}^3 \int_0^{\infty} c_j^2(v_j')^2 dx +Z|v_1(0)|^2\geqq 0,
$$
thus, since $\mathcal F_Z$ is a self-adjoint operator and non-negative we need to have $\sigma(\mathcal F_Z)\subset [0, +\infty)$. Therefore, $n(\mathcal F_Z)=0$.

\item[3)] For $Z<0$, $n(\mathcal F_Z)=1$: From the definition of $\Phi_Z$ above we have immediately that $\Phi_Z\in D(\mathcal F_Z)$ and $\mathcal F_Z\Phi_Z=-\alpha^2\Phi_Z$. Therefore, from item 1) follows that $n(\mathcal F_Z)=1$.

\item[4)] From classical Schr\"odinger theory on the half-line $(0, +\infty)$ for the operator $A=-c^2\frac{d^2}{dx^2}$, with $c > 0$ and Dirichlet-domain $D_{\mathrm{Dir}}=\{f\in H^2(0,+\infty): f(0)=0\}$ we  have that $\sigma_{\mathrm{ess}}(-c\frac{d^2}{dx^2})=\sigma(-c\frac{d^2}{dx^2})=[0, +\infty)$. Thus by Weyl's criterion (see Reed and Simon \cite{RS4}) for all $\lambda\geqq 0$, there exists a sequence $\{\psi_n\}\subset D_{\mathrm{Dir}}$ orthogonal in $L^2(0,+\infty)$ with $\|\psi_n\|_{L^2(0,+\infty)}=1$ such that $\|(A-\lambda I)\psi_n\|_{L^2(0,+\infty)}\to 0$ as $n\to +\infty$.

Next, we note that the self-adjoint operator $\mathcal F_{\mathrm{Dir}}$ with homogeneous Dirichlet boundary conditions
$$
D(\mathcal F_{\mathrm{Dir}})=\{\bold v \in H^2(\mathcal Y): v_1(0-)=v_2(0+)=v_3(0+)=0\},
$$
belongs to the family of self-adjoint extensions of $(\mathcal M, D(\mathcal M))$ (see Proposition \ref{M}). Since $(\mathcal F_{\mathrm{Dir}}, D(\mathcal F_{\mathrm{Dir}}))$ posseses no point spectrum and it is positive definite we  need to have $\sigma_{\mathrm{ess}}(\mathcal F_{\mathrm{Dir}})=\sigma(\mathcal F_{\mathrm{Dir}})\subset[0, +\infty)$. Now, we see that $[0, +\infty)\subset \sigma_{\mathrm{ess}}(\mathcal F_{\mathrm{Dir}})$. Indeed, for $\lambda\geqq 0$ we have that the sequence $\Psi_n=(0, \psi_n, 0)$ satisfies $\{\Psi_n\}\subset D_{\mathrm{Dir}}$, orthogonal in $L^2(\mathcal Y)$, $\|\Psi_n\|_{L^2(\mathcal Y)}=1$ and 
$$
\|(\mathcal F_{\mathrm{Dir}}\Psi_n-\lambda I)\Psi_n\|^2_{L^2(\mathcal Y)}=\|(A-\lambda I)\psi_n\|^2_{L^2(0,+\infty)}\to 0,\;\; as\;\; n\to +\infty,
$$
therefore by the Weyl's criterion follows that $\lambda\in \sigma_{\mathrm{ess}}(\mathcal F_{\mathrm{Dir}})$. Then, by Proposition \ref{esse} we get that all self-adjoint extensions for $(\mathcal M, D(\mathcal M))$ have continuous spectrum being $[0,+\infty)$. Now, since the self-adjoint operators $(\mathcal F_Z, D(\mathcal F_Z))$ may have at most a finite collection of negative eigenvalues, its continuous spectrum, $\sigma_{\mathrm{ac}}(\mathcal F_Z)$, coincides with its essential spectrum $\sigma_{\mathrm{ess}}(\mathcal F_Z)$ and $\sigma(\mathcal F_Z)=\sigma_{\mathrm{ac}}(\mathcal F_Z)\cup \sigma_{\mathrm{pt}}(\mathcal F_Z)$. This finishes the proof.
\end{enumerate}
\end{proof}

The following characterization of the resolvent of the operator $\mathcal A= JE$ will be sufficient for our study here. The results in Theorem \ref{spectrum} will be the main points in the analysis.

\begin{theorem}\label{resol}
Let $Z\in \mathbb R$. For $\lambda\in \mathbb C$ with $-\lambda^2\in \rho(\mathcal F_Z)$, we have that $\lambda$ belongs to the resolvent set of $\mathcal A=JE$ with $D(\mathcal A)=D_{\delta, Z}\times L^2(\mathcal Y)$ and $R(\lambda:\mathcal A)=(\lambda I-\mathcal A)^{-1}: H^1(\mathcal Y)\times L^2(\mathcal Y)\to D(\mathcal A)$ has the representation for $\Psi=(\bold u, \bold v)$
\begin{equation}\label{resolA}
R(\lambda:\mathcal A)\Psi=\left(\begin{array}{c} -R(-\lambda^2: \mathcal F_Z)(\lambda \bold u+\bold v) \\
-\lambda R(-\lambda^2: \mathcal F_Z)(\lambda \bold u+\bold v) -\bold u \end{array}\right),
\end{equation}
where  $R(-\lambda^2: \mathcal F_Z)=(-\lambda^2I_3- \mathcal F_Z)^{-1}: L^2(\mathcal Y)\to D(\mathcal F_Z)$.
\end{theorem}

\begin{remark}
In Remark \ref{Zposi} below we give a explicit formulation for the resolvent operator $R(\eta:\mathcal F_Z)$ for any $\eta<0$ (without loss of generality) and $Z \neq 0$.
\end{remark}

\begin{proof}
For $\Psi=(\bold u, \bold v)\in H^1(\mathcal Y)\times L^2(\mathcal Y)$, $(\mathcal A-\lambda)^{-1}\Psi=(\bold h, \bold j)$ if and only if
\begin{equation}\label{1resolA}
\begin{cases}
I_3\bold u&=-\lambda I_3\bold h+I_3 \bold j \\
I_3\bold v&=-\mathcal F_Z\bold h-\lambda I_3 \bold j,
\end{cases}
\end{equation}
then $(-\lambda^2I_3-\mathcal F_Z)\bold h=I_3(\lambda \bold u+ \bold v)$. Therefore, by hypothesis  we obtain $I_3\bold h=(-\lambda^2I_3-\mathcal F_Z)^{-1}(\lambda \bold u+ \bold v)$. This finishes the proof.
\end{proof}

\begin{remark}\label{Zposi}
For future reference in our study, we establish the resolvent operator $R(-\lambda^2:\mathcal F_Z)$  and $Z\neq 0$. Thus, we start with $Z>0$. From Theorem \ref{spectrum} we obtain for every $\lambda>0$ (without loss of generality) that
for $\bold u=(u_j)_{j=1}^3\in L^2(\mathcal Y)$ and $(\Phi_j)_{j=1}^3=(\mathcal F_Z +\lambda^2I_3)^{-1} \bold u$ the following:
\begin{enumerate}
\item[(a)] for $x<0$
\begin{equation}
\label{formu1}
\Phi_1(x)=(-c_1^2\frac{d^2}{dx^2}+\lambda^2)^{-1}(u_1)(x)=\frac{d_1}{c_1^2}e^{\frac{\lambda}{ c_1 }x} +\frac{1}{2 c_1 \lambda}\int_{-\infty}^0 u_1(y) e^{-\frac{\lambda}{ c_1 } |x-y|} dy
\end{equation}

\item[(b)] for $x>0$ and $j=2,3$,
\begin{equation}
\label{formu2}
\Phi_j(x)=(-c_1^2\frac{d^2}{dx^2}+\lambda^2)^{-1}(u_j)(x)=\frac{d_j}{c_j^2}e^{-\frac{\lambda}{ c_j }x} +\frac{1}{2 c_j \lambda}\int_0^{\infty} u_j(y) e^{-\frac{\lambda}{ c_j } |x-y|} dy,
\end{equation}
where the constants $d_j=d_j(\lambda, (\Phi_j))$ are chosen  such that $(\Phi_j)\in D_{\delta,Z}$. In the following we determine these. So, define
$$
\begin{aligned}
t_1(\lambda)&=\frac{1}{2 c_1 }\int_{-\infty}^0 u_1(y) e^{\frac{\lambda}{ c_1 }y} dy\\
t_j(\lambda)&=\frac{1}{2 c_j }\int_0^{\infty} u_j(y) e^{-\frac{\lambda}{ c_j }y} dy,\quad j=2,3.
\end{aligned}
$$
Then, from the relations for $j=2,3$,
$$
\begin{aligned}
\Phi_1(0-)&=\frac{d_1}{c_1^2}+\frac{1}{\lambda}t_1(\lambda), & \Phi_j(0+) &=\frac{d_j}{c_j^2}+\frac{1}{\lambda}t_j(\lambda),\\
\Phi'_1(0-)&=\frac{d_1\lambda}{c_1^3}-\frac{1}{ c_1 }t_1(\lambda), & \Phi'_j(0+) &=-\frac{d_j\lambda}{c_j^3}+\frac{1}{ c_j }t_j(\lambda),
\end{aligned}
$$
we obtain the linear system
$$
M\left(\begin{array}{c}d_1 \\d_2 \\d_3\end{array}\right)\equiv \left(\begin{array}{ccc}\frac{1}{c_1^2} & -\frac{1}{c_2^2} & 0 \\0 & \frac{1}{c_2^2} & -\frac{1}{c_3^2} \\\frac{1}{ c_1 }+\frac{Z}{c_1^2\lambda} & \frac{1}{ c_2 } & \frac{1}{ c_3 }\end{array}\right)\left(\begin{array}{c}d_1 \\d_2 \\d_3\end{array}\right)=\frac{1}{\lambda}\left(\begin{array}{c}t_2(\lambda)-t_1(\lambda) \\t_3(\lambda)-t_2(\lambda)  \\\sum_{j=1}^3 c_j t_j(\lambda)-\frac{Z}{\lambda}t_1(\lambda)\end{array}\right).
$$
Thus, since $ \det (M)=\frac{1}{(c_1 c_2 c_3)^2}[\sum_{j=1}^3 c_j +\frac{Z}{\lambda}]$ and $Z, \lambda>0$, we obtain the uniqueness of the constants $d_j$ such that $(\Phi_j)\in D_{\delta,Z}$.
\end{enumerate}
Now, for $Z<0$, from Theorem \ref{spectrum} we obtain for every $\lambda>0$ (without loss of generality) and $\lambda^2 \neq -\lambda_0$ with $\lambda_0=-Z^2/(\sum_{j=1}^3 c_j )^2$ that
for $\bold u=(u_j)_{j=1}^3\in L^2(\mathcal Y)$ and $(\Psi_j)_{j=1}^3=(\mathcal F_Z +\lambda^2I_3)^{-1} \bold u$ the following:
\begin{enumerate}
\item[(c)] for $x<0$, $\alpha=-Z/\sum_{j=1}^3 c_j >0$, and  from \eqref{formu1},
\begin{equation}\label{formu3}
\Psi_1(x)=\frac{1}{\lambda^2+\lambda_0}e^{\frac{\alpha}{ c_1 }x} \langle u_1, e^{\frac{\alpha}{ c_1 }x}\rangle +\Phi_1(x),
\end{equation}

\item[(d)] for $x>0$ and  from \eqref{formu2},
\begin{equation}\label{formu4}
\Psi_j(x)=\frac{1}{\lambda^2+\lambda_0}e^{-\frac{\alpha}{ c_j }x} \langle u_j, e^{-\frac{\alpha}{ c_j }x}\rangle +\Phi_j(x),\quad j=2,3,
\end{equation}
here the constants $(d_j)$ in \eqref{formu1}-\eqref{formu2} (unique) are chosen such  that $(\Phi_j)\in D_{\delta,Z}$.

\end{enumerate}
\end{remark}

From Theorem \ref{spectrum} we can define the following equivalent $X^1_Z$-norm to $H^1(\mathcal Y)$, for $\bold v=(v_j)_{j=1}^3\in H^1(\mathcal Y)$
\begin{equation}\label{1norm}
\|\bold v\|_{X^1_Z}^2=\|\bold v'\|^2 _{L^2(\mathcal Y)} + (\beta+1)\|\bold v\|^2 _{L^2(\mathcal Y)} +Z|v_1(0-)|^2,
\end{equation}
where for $Z<0$, $\beta = \frac{Z^2}{9}=\frac{Z^2}{(\sum_{j=1}^3 c_j )^2}$, 
and for $Z\geqq 0$, $\beta=0$. We will denote by $H^1_Z(\mathcal Y)$ the space $H^1(\mathcal Y)$ with the norm $\|\cdot\|_{X^1_Z}$. Moreover, the following well-defined inner product in $H^1_Z(\mathcal Y)$, 
\begin{equation}\label{1inner}
\langle \bold u, \bold v\rangle_{1,Z}= \int_{-\infty}^0 u'_1(x)\overline{v'_1(x)}dx +\sum_{j=2}^3 \int_0^{\infty} u'_j(x) \overline{v'_j(x)}dx +(\beta +1) \langle \bold u, \bold v\rangle + Zu_1(0-)\overline{v_1(0-)},
\end{equation}
induces the $X^1_Z$-norm above (here we are considering $c_j^2 = 1$, without loss of generality).

The following theorem shows that the operator $\mathcal A\equiv JE$ is indeed the infinitesimal generator of a $C_0$-semigroup. To that end, we apply the classical Lumer-Phillips theory. 

\begin{theorem}
\label{cauchy1}
Let $Z\in \mathbb R$ and consider the linear operators $J$ and $E$ defined in \eqref{stat2}. Then, $\mathcal A\equiv JE$ with $D(\mathcal A)= D_{\delta, Z}\times \mathcal E(\mathcal Y)$ is the infinitesimal  generator of a $C_0$-semigroup $\{W(t)\}_{t\geqq 0}$ on $H^1(\mathcal Y)\times L^2(\mathcal Y)$. The initial value problem 
\begin{equation}\label{LW1}
\begin{cases}
\bold w_{t}=\mathcal A\bold w \\
\bold w(0)=\bold w_0\in D(\mathcal A)=D_{\delta, Z}\times \mathcal E(\mathcal Y)
\end{cases}
\end{equation}
has a unique solution $\bold w\in C([0, +\infty): D(\mathcal A))\cap C^1([0, +\infty): H^1(\mathcal Y)\times L^2(\mathcal Y))$ given by $\bold w(t)=W(t)\bold w_0$, $t\geqq 0$.

Moreover, for any $\Psi\in H^1(\mathcal Y)\times L^2(\mathcal Y)$ and $\theta>\beta+1$, $\beta=\frac{Z^2}{(\sum_{j=1}^3 c_j )^2}$, we have the representation formula
\begin{equation}\label{FRW}
W(t)\Psi=\frac{1}{2\pi i}\int_{\theta-i\infty}^{\theta+i\infty} e^{\lambda t} R(\lambda: \mathcal A) \Psi d\lambda
\end{equation}
where $\lambda\in \rho(\mathcal A)$ with $\RE \lambda = \theta$ and $R(\lambda: \mathcal A)=(\lambda I-\mathcal A)^{-1}$, and for every $\delta>0$, the integral converges uniformly in $t$ for every $t\in [\delta, 1/\delta]$.
\end{theorem}
\begin{proof} We divide the proof in several steps (without lost of generality we consider $c_j^2=1$):
\begin{enumerate}
\item[1)] We consider the Hilbert space $X_Z\equiv H^1_Z(\mathcal Y)\times L^2(\mathcal Y)$ with inner product $\langle \cdot, \cdot\rangle_{X_Z}=\langle \cdot, \cdot\rangle_{1,Z} + \langle \cdot, \cdot\rangle$, with $\langle \cdot, \cdot\rangle_{1,Z}$ defined in \eqref{1inner}. Define $\mathcal B=\mathcal A-\gamma I$, $\gamma=\beta+1>0$, with $\beta>0$ as in \eqref{1norm}. Then the following linear initial value problem
\begin{equation}\label{Bcauchy}
\begin{cases}
\bold u_{t}=\mathcal B\bold u \\
\bold u(0)=\bold u_0\in D(\mathcal B)=D_{\delta, Z}\times \mathcal E(\mathcal Y)
\end{cases}
\end{equation}
has a unique solution $\bold u\in C([0, +\infty): D(\mathcal B))\cap C^1([0, +\infty): X_Z)$ given by $\bold u(t)=U(t)\bold u_0$, $t\geqq 0$, where $\{U(t)\}_{t\geqq 0}$ is a $C_0$-semigroup of contractions on $X_Z$. Indeed, The idea is to use the classical Lumer-Phillips Theorem (see, e.g., Pazy \cite{Pa}):
\begin{enumerate}
\item[a)] $\mathcal B$ is dissipative on $X_Z$: for $ \Phi=(\bold u, \bold v)\in D(\mathcal B)$ with $\bold u=(u_j)_{j=1}^3\in D(\mathcal B)$ and $\bold v=(v_j)_{j=1}^3\in \mathcal E(\mathcal Y)$ (with $u_j, v_j$ real-valued without lost of generality)
\begin{equation}
\begin{aligned}
&\langle -\mathcal B\Phi, \Phi\rangle_{X_Z}=\langle -\bold v, \bold u \rangle_{1,Z}+ \langle \mathcal F_Z\bold u, \bold v\rangle +\gamma \|\bold u \|^2_{X^1_Z} +\gamma \|\bold v\|^2_{L^2(\mathcal Y)}\\
&=-\langle \bold v, \bold u \rangle_{1,Z}-Zu_1(0-)v_1(0-) + \int_{-\infty}^0u_1'v_1'dx+\sum_{j=2}^3\int_0^{\infty}u_j'v_j'dx + \gamma \|\bold u \|^2_{X^1_Z} +\gamma \|\bold v\|^2_{L^2(\mathcal Y)}\\
&=(\beta+1)\Big[\|\bold u \|^2_{X^1_Z} + \|\bold v\|^2_{L^2(\mathcal Y)}-\langle \bold u, \bold v \rangle\Big]\geqq 0,
\end{aligned}
\end{equation}
because of the Cauchy-Schwartz inequality and $\|\bold u\|^2_{L^2(\mathcal Y)}\leqq \|\bold u \|^2_{X^1_Z}$ by \eqref{1norm}. 

\item[b)] From Theorem \ref{resol}  we can choose $\lambda$ such that $\lambda+\gamma>0$ and $\lambda+\gamma\in \rho(\mathcal A)$,  then the range, $R(\lambda I-\mathcal B)=R((\lambda+\gamma)I-\mathcal A)$, of  $\lambda I-\mathcal B$ is $X_Z$. Thus we obtain that $\mathcal B$ is the infinitesimal generator of a $C_0$-semigroup of contractions $\{U(t)\}_{t\geqq 0}$ on $X_Z$. Therefore, the solution of the linear problem in \eqref{Bcauchy} is given by $\bold u(u)=U(t)\bold u_0$.
\end{enumerate}
\item[2)] Define $W(t)=e^{\gamma t} U(t)$, then $\{W(t)\}_{t\geqq 0}$ is a $C_0$-semigroup  on $X_Z$ with infinitesimal generator
$$
W'(0)=\gamma I+\mathcal A-\gamma I= \mathcal A.
$$
Then, since the norm $\|\cdot\|_{H^1(\mathcal Y)}$  is equivalent to  the  norm $\|\cdot\|_{X_Z^1}$ on $H^1(\mathcal Y)$, we obtain that $\{W(t)\}_{t\geqq 0}$ is a $C_0$-semigroup  on $H^1(\mathcal Y)\times L^2(\mathcal Y)$ and $\bold w(t)=W(t)\bold w_0$ is the unique solution for the linear problem \eqref{LW1}.
\item[3)] From item 2) we have $\|W(t)\|_{H^1(\mathcal Y)\times L^2(\mathcal Y)}=e^{\gamma t} \|U(t)\|_{H^1(\mathcal Y)\times L^2(\mathcal Y)}\leqq M e^{\gamma t} \|U(t)\|_{X_Z}\leqq M e^{\gamma t}$, for $M>0$, $\gamma=\beta+1>0$ and $t\geqq 0$. Therefore, from Theorem \eqref{resol}, the semigroup theory, and the Laplace transform we obtain for  $\theta>\beta+1$ the representation formula in \eqref{FRW}. This finishes the proof.
\end{enumerate}
\end{proof}

The following proposition simply states the (expected) invariance property of the energy space under the action of the semigroup.
\begin{proposition}\label{preser}
The semigroup $\{W(t)\}_{t\geqq 0}$ defined by formula \eqref{FRW} left invariance the subspace $\mathcal E(\mathcal Y)\times L^2(\mathcal Y)$. Moreover, $W(t)(\mathcal E(\mathcal Y)\times L^2(\mathcal Y))\subset \mathcal E(\mathcal Y)\times \mathcal C(\mathcal Y)$, $t>0$, where
\begin{equation}\label{0continuity}
\mathcal C(\mathcal Y)=\{(v_j)_{j=1}^3\in  L^2(\mathcal Y): v_1(0-)=v_2(0+)=v_3(0+)\}.
\end{equation}
\end{proposition}
\begin{proof} By the representation of $W(t)$ in \eqref{FRW} is sufficient to show that the resolvent operator $R(\lambda:\mathcal A)\Phi\in \mathcal E(\mathcal Y)\times L^2(\mathcal Y)$ for $\Phi\in \mathcal E(\mathcal Y)\times L^2(\mathcal Y)$. Indeed, for $\Psi=(\bold u, \bold v)$ we have from \eqref{resolA} that $R(-\lambda^2: \mathcal F_Z)(\lambda \bold u+\bold v)\in D_{\delta, Z}\subset \mathcal E(\mathcal Y)$ and so $R(\lambda:\mathcal A)\Phi\in \mathcal E(\mathcal Y)\times \mathcal E(\mathcal Y)\subset \mathcal E(\mathcal Y)\times \mathcal C(\mathcal Y)\subset \mathcal E(\mathcal Y)\times L^2(\mathcal Y)$.
\end{proof}

Our local well-posedness result for the sine-Gordon equation on a $\mathcal Y$-junction with a $\delta$-interaction is the following.

\begin{theorem}\label{well0} 
	For any $\Psi \in \mathcal
	E(\mathcal{Y})\times L^2(\mathcal Y) $
	there exists $T > 0$ such that the sine-Gordon
	equation \eqref{stat1} has a unique solution $\mathbf w \in C
	([0,T]; \mathcal E(\mathcal{Y})\times  L^2(\mathcal Y) )$ satisfying  $\mathbf
	w(0)=\Psi$. For
	each $T_0\in (0, T)$ the mapping 
	\begin{equation}\label{mapping}
	\Psi\in \mathcal E (\mathcal{Y})\times  L^2(\mathcal Y) \to \mathbf w \in C ([0,T_0];
	\mathcal
	E(\mathcal{Y})\times  L^2(\mathcal Y) ),
	\end{equation}
	 is at least of class $C^2$. Moreover, for all $t>0$, $\mathbf w(t) \in \mathcal E(\mathcal{Y})\times  \mathcal C(\mathcal Y) $.
\end{theorem}
\begin{proof} Based on Theorem \ref{cauchy1} and Proposition \ref{preser}
	the local well-posedness result in $X(\mathcal{Y})\equiv \mathcal E(\mathcal{Y})\times  L^2(\mathcal Y)$ follows from standard arguments of the Banach fixed point theorem. We will give the
	sketch of the proof for convenience of the reader. Consider the mapping $J_{\Psi}:
	C([0, T]: X(\mathcal{Y}))\longrightarrow C([0,T];X( \mathcal Y))$ given by
	$$
	J_{\Psi}[\mathbf w](t)=e^{t\mathcal A}\Psi+ \int
	_0^t e^{(t-s) \mathcal A}F(\mathbf w(s))ds,
	$$
	where $e^{t\mathcal A}$ is the $C_0$-semigroup $\{W(t)\}_{t\geqq 0}$ defined in \eqref{FRW}.
One needs to show that the mapping $J_{\Psi}$ is well-defined. We note immediately  that the nonlinearity satisfies for $\mathbf w=(\bold u, \bold v)\in X(\mathcal{Y})$ that $F(\mathbf w)\in   \mathcal E (\mathcal{Y})\times  \mathcal C(\mathcal Y)\subset \mathcal E (\mathcal{Y})\times  L^2(\mathcal Y)$ with $\| F(\mathbf w)\|_{X(\mathcal{Y})}\leqq \|\bold u\|_{L^2(\mathcal Y)}\leqq \|\mathbf w\|_{X(\mathcal{Y})}$. Thus we obtain for $t\in [0,T]$
\[
\| J_\Psi[\mathbf w](t)\| _{X(\mathcal{Y})}\leq Me^{\gamma T}\|\Psi\|_{X(\mathcal{Y})}+\frac{M}{\gamma} \Big(e^{\gamma T}-1\Big)
\sup\limits_{s\in[0,T]}\| \mathbf  w(s)\| _{X(\mathcal{Y})},
\]
where the positive constants $M, \gamma$ do not depend on $\Psi$ and are determined by the semigroup $W(t)$ (see Theorem \ref{cauchy1} and its proof). The
	continuity and contraction property of $J_{\Psi}$ are proved
	in a standard
	way. Therefore, we obtain the existence of a unique solution to
	the Cauchy problem
	associated to  \eqref{stat1}  on $\mathcal E(\mathcal{Y})\times  L^2(\mathcal Y)$ and that the mapping data-solution in \eqref{mapping} is at least continuous.
	
Next, we recall that the argument based on the contraction
	mapping principle above
	has the advantage that if $F(\mathbf w)$ has a specific regularity, 
	then it is inherited by the mapping data-solution. In particular,
	following the ideas in \cite{AngGol18b}, we consider
	for $(\Psi, \mathbf  z)\in B(\Psi;\epsilon)\times C([0, T], X(\mathcal Y)) $ the
	mapping
	$$
	\Gamma(\Psi, \mathbf z  )(t)=\mathbf  z(t)- J_{\Psi}[\mathbf  z](t),\qquad
	t\in [0, T].
	$$
	Then $\Gamma(\Psi, \mathbf  w)(t)=0$ for all $t\in [0, T]$, and since $F(\mathbf z)$ is smooth we obtain that $\Gamma$ is
	smooth. Hence, using the arguments applied for obtaining the
	local well-posedness
	in $X(\mathcal Y)$ above, we can show that the operator
	$\partial_\mathbf 
	z\Gamma(\Psi, \mathbf  w)$ is one-to-one and onto. Thus, by the
	Implicit Function
	Theorem there exists a smooth mapping $\mathbf  \Lambda: B(\Psi;\delta)\to C([0,
	T], X(\mathcal Y))$ such that $\Gamma(\mathbf  V_0, \mathbf 
	\Lambda (\mathbf  V_0))=0$
	for all $\mathbf  V_0\in B(\Psi;\delta)$. This argument establishes
	the smoothness
	property of the mapping data-solution associated to the sine-Gordon equation.
	Lastly, from the Proposition \ref{preser} and from the arguments above we obtain that for all $t>0$, $\mathbf w(t) \in \mathcal E(\mathcal{Y})\times  \mathcal C(\mathcal Y) $. This finishes the proof.
\end{proof}

 \section{Linear instability criterion for the sine-Gordon model on a $\mathcal Y$-junction}
\label{seccriterium}

In this section we establish a linear instability criterion of stationary solutions for the sine-Gordon model \eqref{sg2} on a  $\mathcal{Y}$-junction. The analytical criterion developed here applies to both the typical $\mathcal Y$-junction of type I , (see Figure \ref{figYtipoI}), and of type II (see Figure \ref{figYtipoII}). More importantly, the criterion also applies to any type of stationary solutions independently of the boundary conditions under consideration and can be therefore used to study configurations with boundary rules at the vertex of $\delta'$-interaction type, or with other types of stationary solutions to the sine-Gordon equation such as breathers, for instance.

Let us suppose that $JE$ on a domain $D(JE)\subset L^2(\mathcal Y)$ is the infinitesimal generator of a $C_0$-semigroup on $L^2(\mathcal Y)$ and the stationary solution $\Phi=(\phi_1(x), \phi_2(x), \phi_3(x),0,0,0)\in D(JE)$. Thus, every component satisfies the equation
\begin{equation}\label{statio}
-c^2_j \phi''_j + \sin (\phi_j)=0,\quad j=1,2,3.
\end{equation}
 Now, we suppose  that $\bold w$ satisfies formally  equality in \eqref{stat1} and we define
  \begin{equation}\label{stat3}
 \bold v \equiv \bold w-\Phi,
  \end{equation}
  then, from \eqref{statio} we obtain (by using the notation) the following  linearized system for \eqref{stat1} around $\Phi$. 
   \begin{equation}\label{stat4}
  \bold v_t=J\mathcal E\bold v
 \end{equation} 
  with $\mathcal E$ being the $6\times 6$ diagonal-matrix $\mathcal E=\left(\begin{array}{cc} \mathcal L & 0 \\0 & I_3\end{array}\right)$, and 
 \begin{equation}\label{stat5}
\mathcal{L} =\Big (\Big(-c_j^2\frac{d^2}{dx^2}+ \cos (\phi_j)
\Big)\delta_{j,k} \Big ),\qquad 1\leqq j, k\leqq 3.
\end{equation}

We point out the equality $J\mathcal E=J E+\mathcal{T}$, with 
$$
\mathcal{T} = \left(\begin{array}{cc} 0& 0 \\ \big( - \cos(\phi_j) \, \delta_{j,k} \big)& 0\end{array}\right)
$$
 being a \emph{bounded} operator on $H^1(\mathcal G)\times L^2(\mathcal G)$. This  implies that $J\mathcal E$ also generates a  $C_0$-semigroup on  $H^1(\mathcal G)\times L^2(\mathcal G)$ (see Pazy \cite{Pa}).
  
In the sequel, our objective is to provide sufficient conditions for the trivial solution $\bold v \equiv 0$ to be unstable by the linear flow \eqref{stat4}. More precisely, we are interested in finding a {\it growing mode solution} of \eqref{stat4} with the form $ \bold v=e^{\lambda t} \Psi$ and $\RE \lambda >0$. In other words, we need to solve the formal system 
 \begin{equation}\label{stat6}
 J\mathcal E \Psi=\lambda \Psi,
\end{equation}
with $\Psi\in D(J\mathcal E)$. If we denote by $\sigma(J\mathcal E)= \sigma_{\mathrm{pt}}(J\mathcal E)\cup \sigma_{\mathrm{ess}}(J\mathcal E)$ the spectrum  of $J\mathcal E$ (namely, $\lambda \in  \sigma_{\mathrm{pt}}(J\mathcal E)$ if $\lambda$ is isolated and with finite multiplicity), the later discussion
suggests the usefulness of the following definition:
\begin{definition}
The stationary vector solution $\Phi \in D(\mathcal E)$    is said to be \textit{spectrally stable} for model sine-Gordon if the spectrum of $J\mathcal E$, $\sigma(J\mathcal \mathcal E)$, satisfies $\sigma(J\mathcal E)\subset i\mathbb{R}.$
Otherwise, the stationary solution $\Phi\in D(\mathcal E)$   is said to be \textit{spectrally unstable}.
\end{definition}

It is standard to show that $ \sigma_{\mathrm{pt}}(J\mathcal E)$ is symmetric with respect  to both the real and imaginary axes and $ \sigma_{\mathrm{ess}}(J\mathcal E)\subset i\mathbb{R}$ by supposing $J$ skew-symmetric and $\mathcal E$ self-adjoint   (by supposing, by instance, Assumption $(S_3)$ below for $\mathcal L$ (see \cite[Lemma 5.6 and Theorem 5.8]{GrilSha90}). These cases on $J$ and  $\mathcal E$ will be considered in our theory. Hence  it is equivalent to say that $\Phi\in D(J\mathcal E)$ is  \textit{spectrally stable} if $ \sigma_{\mathrm{pt}}(J\mathcal E)\subset i\mathbb{R}$, and it is spectrally unstable if $ \sigma_{\mathrm{pt}}(J\mathcal E)$ contains point $\lambda$ with  $\RE \lambda>0.$
 
 From \eqref{stat6}, our eigenvalue problem to solve is now reduced to,
  \begin{equation}\label{stat11}    
 J\mathcal E\Psi =\lambda \Psi, \quad \RE \lambda >0,\;\;\Psi\in D(\mathcal E).
 \end{equation}  
 Next, we establish our theoretical framework and assumptions for obtaining a nontrivial solution to  problem in \eqref{stat11}:
 \begin{enumerate}
 \item[($S_1$)]  $J\mathcal E$ is the generator of a $C_0$-semigroup $\{S(t)\}_{t\geqq 0}$.  
 \item[($S_2$)] Let $\mathcal L$ be the matrix-operator in \eqref{stat5}  defined on a domain $D(\mathcal L)\subset L^2(\mathcal G)$ on which $\mathcal L$ is self-adjoint.
 \item[($S_3$)] Suppose $\mathcal L:D(\mathcal L)\to L^2(\mathcal G)$ is  invertible  with Morse index $n(\mathcal L)=1$ and such that $\sigma(\mathcal L)=\{\lambda_0\}\cup J_0$ with $J_0\subset [r_0, +\infty)$, for $r_0>0$, and $\lambda_0<0$,
 \end{enumerate} 
 
Our linear instability criterion is the following.
\begin{theorem}\label{crit}
Suppose the assumptions $(S_1)$ - $(S_3)$ hold.  Then the operator $J\mathcal E$ has a real positive and a real negative eigenvalue. 
\end{theorem}

The proof of Theorem \ref{crit} is based in ideas from Lopes \cite{Lopes} and from
the following result on closed convex cone (see Krasnoelskii \cite{Kra}, Chapter 2, section 2.2.6).

\begin{theorem}\label{Kra}
Let $K$ be a closed convex cone of a Hilbert space $(X,\|\cdot\|)$ such that there are a continuous linear functional $\Pi$ and a constant $a>0$ such that $\Pi(u)\geqq a\|u\|$ for any $u\in K$. If $T:X\to X$ is a bounded linear operator  that leaves $K$ invariant, then $T$ has an eigenvector in $K$ associated to a nonnegative eigenvalue.
\end{theorem}
\begin{proof}[Proof of Theorem \ref{crit}]
 From assumption $(S_1)$ we have  that $J\mathcal E$ is the infinitesimal generator of a $C_0$-semigroup $\{S(t)\}_{t\geqq 0}$.  For $\psi_0\in D(\mathcal J)$, $\|\psi_0\|=1$ and $\lambda_0<0$ such that $\mathcal L\psi_0=\lambda_0 \psi_0$, we consider $\Psi_0=(\psi_0,0,0,0)^\top$ and  the following nonempty closed convex cone 
 \begin{equation*}
K_0=\{z\in D(\mathcal E): \langle \mathcal Ez, z\rangle \leqq 0,\;\;\text{and}\;\; \langle z, \Psi_0\rangle\geqq 0 \}.
\end{equation*}
Next, we see that $K_0$ is invariant by the semigroup $S(t)$.  Indeed, we will use a density argument based in the existence of a core for $\mathcal A\equiv J\mathcal E$. Thus, from semi-group theory follows that the space
 $$
 D(\mathcal A^\infty)=\bigcap_{n\in \mathbb N} D(\mathcal A^n)
 $$
 with $D(\mathcal A^n)=\{f\in D(\mathcal A^{n-1}): \mathcal A^{n-1}f\in D(\mathcal A)\}$, result to be dense in $L^2(\mathcal G)$ and it is a $\{S(t)\}_{t\geqq 0}$-invariant subspace of $D(\mathcal A)$. Thus, $D(\mathcal A^\infty)$ is a core for $\mathcal A$.  Therefore
is enough to consider the case $f\in K_0\cap D(\mathcal A^\infty)$ and so the Hamiltonian equation
 \begin{equation}
\left\{ \begin{array}{ll}
\dot{z}=\mathcal A z\\
z(0)= f
\end{array} \right.
 \end{equation}
 has solution $z(t)=S(t)f\in D(\mathcal A^\infty)$ and therefore from the self-adjoint property of $\mathcal E$ and the skew-symmetric property of $J$ we obtain
 $$
 \frac{d}{dt}  \langle \mathcal E z(t), z(t)\rangle=  \langle \mathcal EJ\mathcal Ez(t), z(t)\rangle +  \langle \mathcal Ez(t), J\mathcal E z(t)\rangle =0,
 $$
 then for all $t$, $ \langle \mathcal Ez(t), z(t)\rangle= \langle \mathcal Ef, f\rangle \leqq 0$. Next, we suppose $\langle f, \Psi_0\rangle> 0$ and that there is $t_0$ such that $\langle S(t_0) f, \Psi_0\rangle<0$. Then by continuity of the flow $t\to S(t)f$ there is $\tau\in (0, t_0)$ with $\langle S(\tau)f, \Psi_0\rangle=0$. Now, from  assumption $(S_4)$
  we have from the spectral theorem applied to the self-adjoint operator $\mathcal E$ (more specifically to $\mathcal L$), the orthogonal decomposition for $f_\tau=S(\tau)f,$
  $$
 f_\tau =\sum_{i=1}^m a_i h_i +g,\quad g\bot h_i, \;\; \text{for all}\;\; i,
 $$
 where $\mathcal E h_i=\lambda_i h_i$, $\|h_i\|=1$, $\lambda_i\in \sigma_{\mathrm{pt}}(\mathcal E)$ with $\lambda_i\geqq \eta$, and $\langle \mathcal Eg,g\rangle\geqq \theta \|g\|^2$, $\theta>0$.
 Therefore,
 $$
 0\geqq  \langle \mathcal E f_\tau, f_\tau\rangle\geqq \sum_{i=1}^m a^2_i\lambda_i +\theta \|g\|^2\geqq \eta \sum_{i=1}^m a^2_i+\theta \|g\|^2\geqq 0.
 $$
 Thus, it follows $g=0$ and $a_i=0$ for $i$. Therefore, $S(\tau)f=0$ and since $S(t)$ is a semigroup we obtain $f=0$ and so $\langle f, \Psi_0\rangle= 0$ which is a contradiction. Now we suppose $\langle f, \Psi_0\rangle= 0$, then the former analysis shows  $f=0$ and so $S(t)f\equiv 0$ for all $t$. It shows the invariance of $K_0$ by $S(t)$. Then, for $\mu$ large we obtain from semigroup's theory the integral representation of the resolvent  
 \begin{equation}\label{lapla}
 Tz=(\mu I-\mathcal A)^{-1}(z)=\int_0^\infty e^{-\mu t} S(t)z dt
  \end{equation}
 and it also leaves $K_0$ invariant. Next, for  $\Pi: L^2(\mathcal G)\to \mathbb R$ defined by $\Pi(z)= \langle z, \Psi_0 \rangle$ we will see that there is $a>0$ such that $\Pi(z)\geqq a\|z\|$ for any $z\in K_0$. Indeed, suppose for $\|g\|=1$, $\langle g, \Psi_0 \rangle=\gamma>0$  and  $\langle \mathcal E g, g\rangle\leqq 0$. Since $ \ker (\Pi)$ is a hyperplane we obtain $g=z+\gamma \Psi_0$ with $\langle z, \Psi_0 \rangle=0$.  So, $-\lambda_0\gamma^2\geqq  \langle \mathcal E z, z\rangle$. Now, from the orthogonal decomposition 
  $
 z =\sum_{i=1}^m  \langle z, h_i \rangle h_i +g,\quad g\bot h_i, \;\; \text{for all}\;\; i,
 $
follows for $\eta, \theta>0$,
$\langle \mathcal E z, z\rangle= \min \{\eta, \theta\}(1-\gamma^2)$. Then,
$$
\langle g, \Psi_0 \rangle=\gamma\geqq \sqrt{\frac{ \min \{\eta, \theta\}}{-\lambda_0+ \min \{\eta, \theta\}}}\equiv a.
$$
Therefore,  by the analysis above and Theorem \ref{Kra}, we conclude that there exist an $\alpha\geqq 0$ and a nonzero element $\omega_0\in K_0$ such that $
 (\mu I -\mathcal A)^{-1}(\omega_0)=\alpha \omega_0$. It is immediate that $\alpha>0$ and so $J\mathcal E\omega_0 = \zeta \omega_0$ with $\zeta=\frac{\mu \alpha-1}{\alpha}$. Next we see that $\zeta\neq 0$. Suppose that $\zeta= 0$, then  the injectivity of $J$ and $\mathcal E$ implies 
 $\omega_0=0$, which is a contradiction. Then, $J\mathcal E$ has a nonzero real eigenvalue $\zeta$.
 
 Now, we have $\sigma(J\mathcal E)=\sigma((J\mathcal E)^*)=-\sigma(\mathcal E J)=
 -\sigma(\mathcal E J \mathcal E \mathcal E^{-1})=-\sigma(J\mathcal E)$ and so $-\zeta$ also belongs to  $\sigma(J\mathcal E)$. Thus from Theorem 5.8  of \cite{GrilSha90}, the essential spectrum of $J\mathcal E$ lies on the imaginary axis and therefore $-\zeta$ is an eigenvalue of $J\mathcal E$ and this finishes the proof.
\end{proof}

\section{Instability of stationary solutions of kink type  for the sine-Gordon equation with $\delta$-interaction}
\label{secInst}

In this section apply the linear instability criterion (Theorem \ref{crit} above) to the case of stationary solutions of kink type determined by a $\delta$-interaction type at the vertex  $\nu=0$. First we examine the structure of such stationary wave solutions.

\subsection{Stationary solutions for the sine-Gordon equation with $\delta$-interaction}
\label{secprofiles}

Next we consider the sine-Gordon model \eqref{sg2} on a  $\mathcal Y$-junction type of the form $\mathcal Y=(-\infty, 0)\cup (0, +\infty)\cup (0, +\infty)$. The stationary solutions will be of the  type in \eqref{trav20} and satisfying the system in \eqref{trav22}. Suppose that the profile $\Phi =(\phi_j)_{j=1}^3$ belongs to the domain $ D(\mathcal L_Z)$ of  $\delta$-type defined in \eqref{2trav23}. Let us rewrite the family of stationary profiles \eqref{trav22} here for convenience:
\begin{equation}
\label{trav24}
\begin{split}
&\phi_1(x)=4  \arctan \Big(e^{\frac{1}{ c_1 }(x-a_1)}\Big),\;\;\;\;\;x<0,\\
& \phi_i(x)=4  \arctan \Big(e^{-\frac{1}{ c_i }(x-a_i)}\Big), \;\;\;\; x>0,\;\;i=2,3.
\end{split}
\end{equation}
Then, clearly, from the continuity condition at the vertex $\nu=0$ we obtain $\frac{1}{ c_1 }a_1=-\frac{1}{ c_2 }a_2=-\frac{1}{ c_3 }a_3$. The Kirchhoff-type  condition in \eqref{2trav23} implies the following relation for $a_1$,
\begin{equation}\label{trav25}
-\frac{e^{-\frac{a_1}{ c_1 }}}{1+e^{-\frac{2a_1}{ c_1 }}}\sum\limits_{j=1}^3  c_j =Z \arctan \Big(e^{-\frac{a_1}{ c_1 }}\Big).
\end{equation}
From \eqref{trav25} we deduce that $Z<0$. Next, from the behavior of the function
\begin{equation}\label{trav26}
f(y)=\frac{1+y^2}{y}  \arctan (y), \;\;y\geqq 0
\end{equation}
we obtain that $Z\in (-\sum_{j=1}^3  c_j , 0)$.  Moreover, we have the following specific behavior of the profiles $ \phi_i$:
\begin{enumerate}
\item[1)] for  $Z\in (-\sum_{j=1}^3  c_j , -\frac{2}{\pi} \sum_{j=1}^3  c_j )$ we obtain $a_1>0$; therefore $a_2, a_3<0$, $ \phi_i'' >0$ for every $i$, and $\phi_1' >0$, $ \phi_j' <0$ ($j=1,2$). Thus, the profile of $(\phi_j)_{j=1}^3$  should look similar to Figure \ref{figKTail} above (tail-profile). Moreover, $\phi_i\in (0, \pi)$, $i=1,2, 3$.

\item[2)] for   $Z\in (-\frac{2}{\pi} \sum_{j=1}^3  c_j , 0)$ we obtain $a_1<0$; therefore $a_2, a_3>0$, and $\phi_i''(a_i) =0$,  $i=1,2, 3$. We also have $\phi_1' >0$, $ \phi_i' <0$ ($i=1,2$). Thus, the profile of $(\phi_j)_{j=1}^3$  should look similar to Figure \ref{figKBump} below (bump-profile). Moreover, $\phi_i\in (0, \eta_0)$, $i=1,2, 3$, $\eta_0=4  \arctan \Big(e^{-\frac{1}{ c_1 }a_1}\Big)>\pi$,

\item[3)]  the case  $Z=-\frac{2}{\pi} \sum_{j=1}^3  c_j $ implies $a_1=0=a_2=a_3$; therefore,  $\phi_i(0) =\pi$ and   $\phi_i''(0) =0$, $i=1,2, 3$. In this case, we have a ``smooth'' profile around the vertex $\nu=0$ (see Figure \ref{figKSmooth} below). 
\end{enumerate}

\begin{figure}[t]
\begin{center}
\subfigure[$Z\in (-\sum_{j=1}^3  c_j , -\frac{2}{\pi} \sum_{j=1}^3  c_j )$]{\label{figKTail}\includegraphics[scale=.42, clip=true]{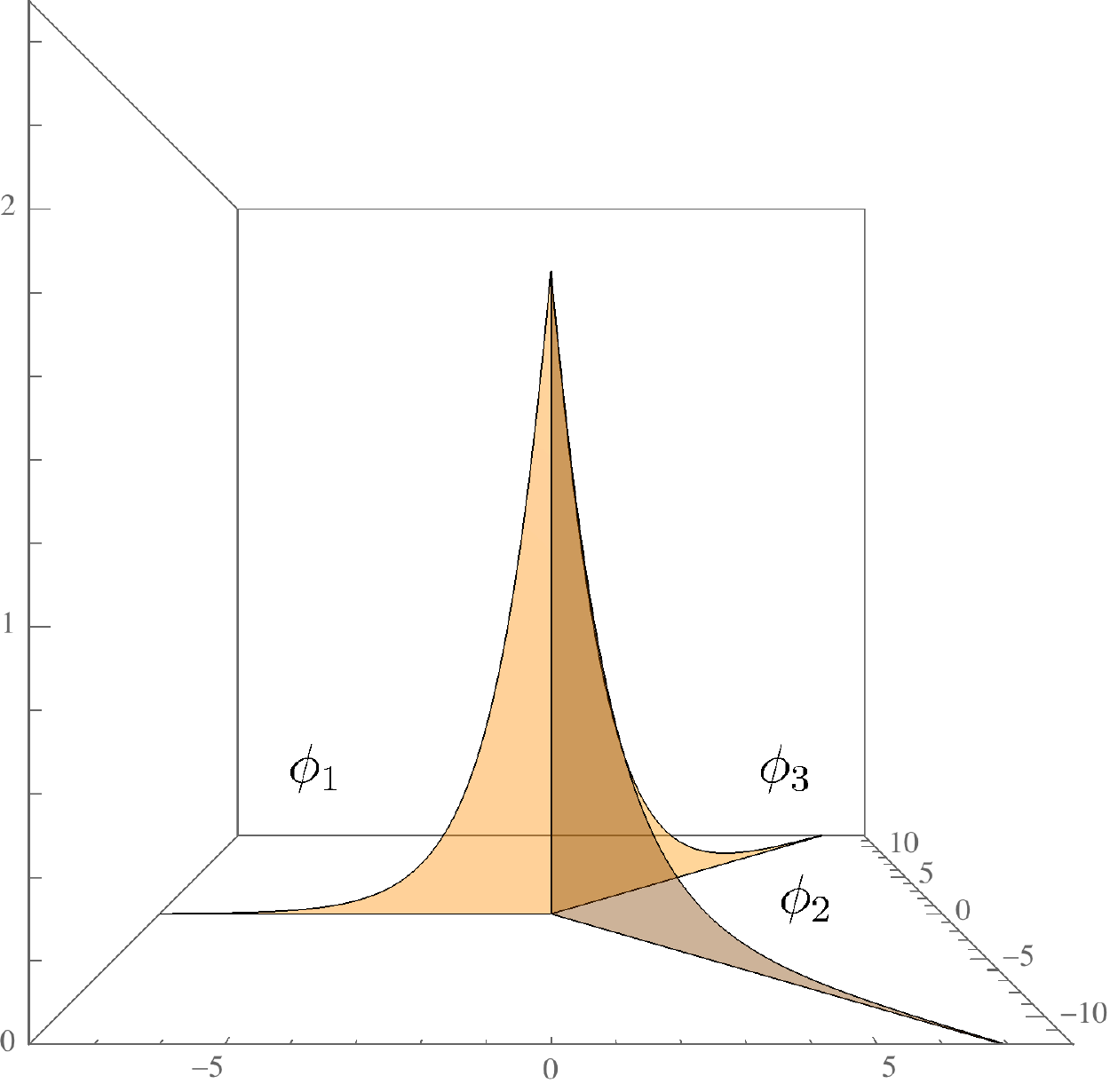}}
\subfigure[$Z\in (-\frac{2}{\pi} \sum_{j=1}^3  c_j , 0)$]{\label{figKBump}\includegraphics[scale=.42, clip=true]{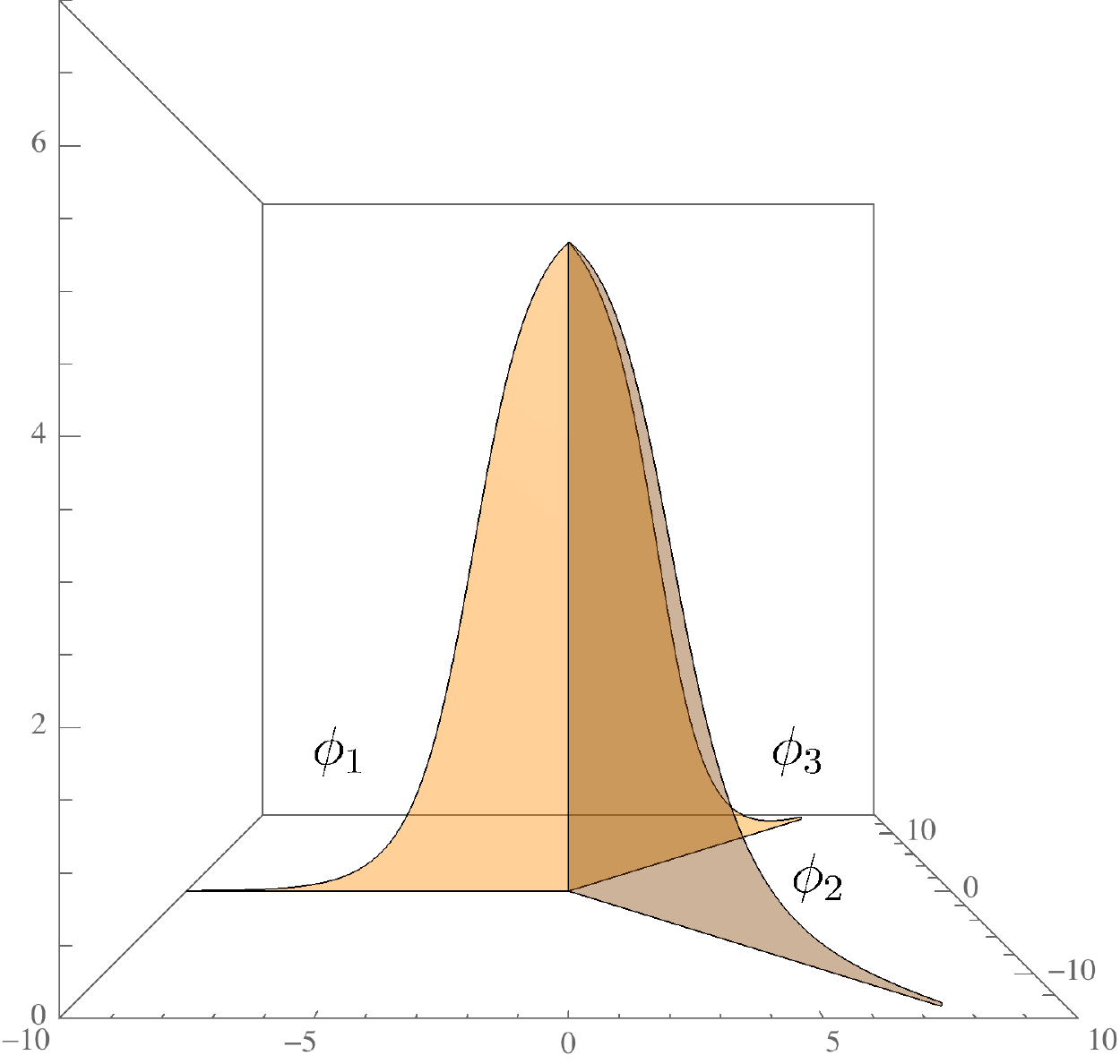}}
\subfigure[$Z=-\frac{2}{\pi} \sum_{j=1}^3  c_j $]{\label{figKSmooth}\includegraphics[scale=.42, clip=true]{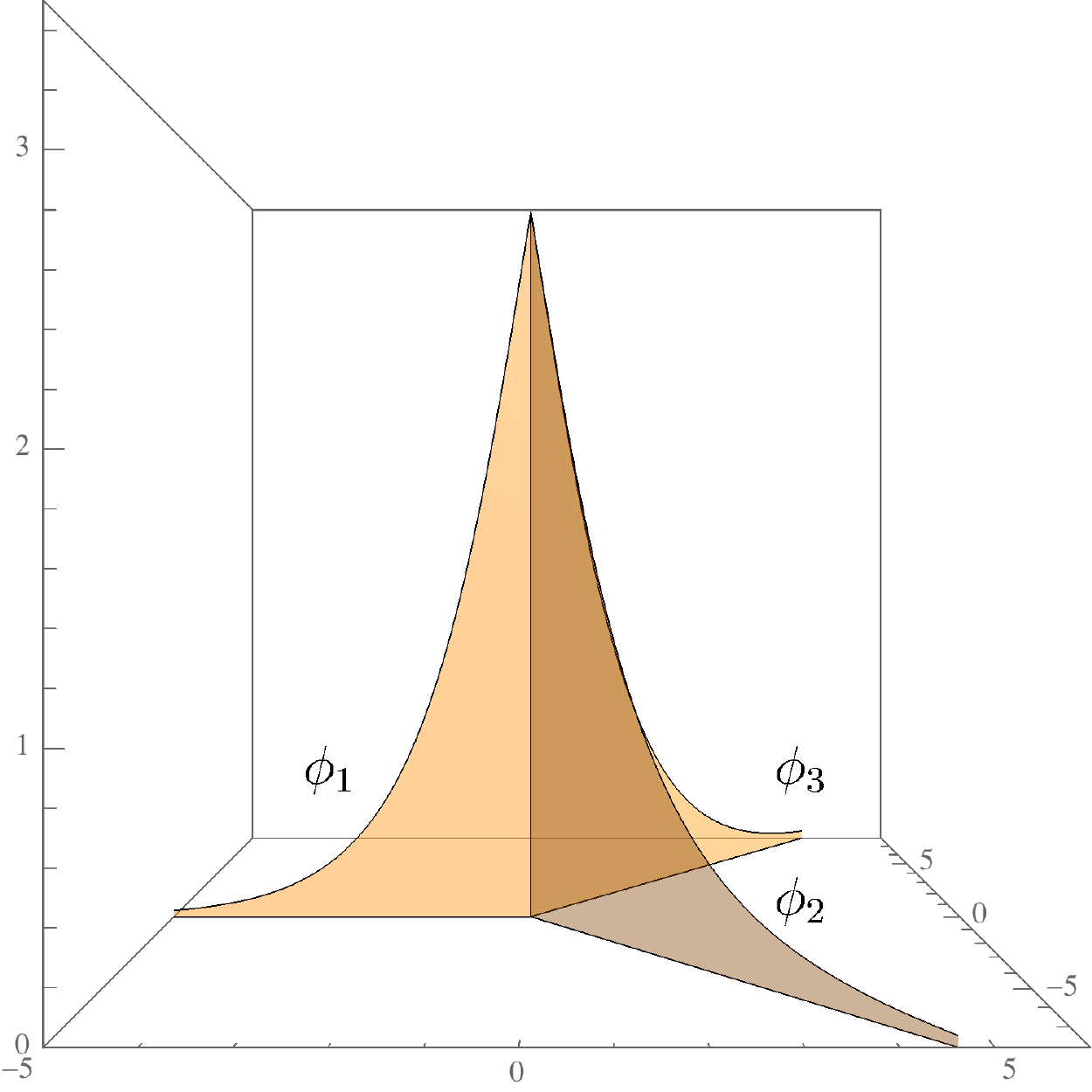}}
\end{center}
\caption{\small{Plots of stationary solutions \eqref{trav24} in the case where $c_j = 1$ for all $j=1,2,3$, for different values of $Z\in  (-\sum_{j=1}^3  c_j , 0) = (-3,0)$. Panel (a) shows the stationary profile solution (``tail'' configuration) for the value $Z = -5/2 \in (-3,-6/\pi)$. Panel (b) shows the profile of ``bump'' type for the value $Z = -1/6 \in (-6/\pi,0)$. Panel (c) shows the ``smooth'' profile solution when $Z = -6/\pi$ (color online).}}\label{figKsolitons}
\end{figure}

We shall see that the stability study of the profiles given in \eqref{trav24} will  require exactly the three cases considered above. The main stability result for the stationary profiles
$\Phi_Z=(\phi_1, \phi_2, \phi_3, 0,0,0)$ with $\phi_j$ defined in \eqref{trav24}-\eqref{trav25} is the following.

\begin{theorem}
\label{1main}
Let $Z\in  (-\sum_{j=1}^3  c_j , 0)$ and the smooth family of stationary profiles $Z\to \Phi_Z$ defined in \eqref{trav24}-\eqref{trav25}. Then $\Phi_Z$ is spectrally unstable for the sine-Gordon model \eqref{stat1}.
\end{theorem}

The proof of spectral instability result in Theorem \ref{1main} will be a consequence of Theorem \ref{crit}. Thus by Theorem \ref{cauchy1} we only need to verify the assumptions $(S_1)$, $(S_2)$ and $(S_3)$ associated to the operator  in \eqref{stat5} with domain $D(\mathcal L_Z)$ of $\delta$-type defined in \eqref{domainZ}.  That will be the focus of the following subsection. 

It is widely known  that the spectral instability of a specific traveling wave solution of an evolution type model is  a key prerequisite to show their nonlinear instability property (see \cite{GrilSha90, Lopes, ShaStr00} and references therein). Thus we have the following definition.
\begin{definition}\label{nonstab}
A stationary vector solution $\Phi \in D(\mathcal E)$ is said to be \textit{nonlinearly unstable} in $X\equiv H^1(\mathcal Y)\times L^2(\mathcal Y)$-norm for model sine-Gordon if there is $\epsilon>0$, such that for every $\delta>0$ there exist an initial datum $\bold w_0$ with $\|\Phi -\bold w_0\|_X<\delta$ and an instant $t_0=t_0(\bold w_0)$ such that $\|\bold w(t_0)-\Phi \|_X>\epsilon$, where $\bold w=\bold w(t)$ is the solution of the sine-Gordon model with $\bold w(0)=\bold w_0$.
\end{definition}
Therefore, the nonlinear instability property of $\Phi_Z$ will be a consequence of Theorem \ref{well0} and the approach by Henry {\it et al.} in \cite{HPW82} (see Remark \ref{remnolineal} below).

\subsection{Spectral study in the case of a $\delta$-interaction}
\label{secspecstudy}

In this section we study the structure of the kernel and Morse index of the following diagonal-matrix Schr\"odinger operator
\begin{equation}\label{spec1}
\mathcal{L}_Z =\Big (\Big(-c_j^2\frac{d^2}{dx^2}+ \cos (\phi_j)
\Big)\delta_{j,k} \Big ),\quad 1\leqq j, k\leqq 3
\end{equation}
with domain $D(\mathcal{L}_Z )$ defined in \eqref{domainZ} and $\phi_j$ given by \eqref{trav24}-\eqref{trav25}. From Proposition \ref{M} in Appendix we have that $(\mathcal{L}_Z, D(\mathcal{L}_Z))$ represents a family of self-adjoint operators.

\begin{proposition}\label{main2}
Let $Z\in  (-\sum_{j=1}^3  c_j , 0)$. Then $ \ker ( \mathcal{L}_Z )=\{\mathbf{0}\}$.
\end{proposition}

\begin{proof} Let $\bold{u}=(u_1, u_2, u_3)\in D(\mathcal{L}_Z )$ and $\mathcal{L}_Z \bold{u}=\bold{0}$. Since $-c_j^2\frac{d^2}{dx^2}\phi'_j+ \cos (\phi_j)\phi'_j=0$, $j=1,2,3$, we obtain from Sturm-Liouville theory on half-lines that
\begin{equation}\label{spec2}
u_1(x)=\alpha_1\phi'_1(x), \;\;x<0,\quad u_j(x)=\alpha_j\phi'_j(x), \;\;x>0,\;\; j=2,3,
\end{equation}
with $\alpha_j\in \mathbb R$. Therefore, since $\phi'_j(0+)=-\phi'_1(0-)\frac{ c_1 }{ c_j }$, $j=2,3$, we obtain from the conditions given by elements of $D(\mathcal{L}_Z )$ that
\begin{equation}\label{spec3}
\alpha_1=-\alpha_2\frac{ c_1 }{ c_2 }=-\alpha_3\frac{ c_1 }{ c_3 },\;\;\;\;
 \sum\limits_{j=2}^3 \alpha_j c_j^2\phi''_j(0+) - \alpha_1c_1^2\phi''_1(0-)=Z\alpha_1\phi'_1 (0-).
\end{equation}
Next, we consider the following cases:
\begin{enumerate}
\item[1)] Let $Z=-\frac{2}{\pi} \sum_{j=1}^3  c_j $. Then, from $\phi''_j(0)=0$ for all $i$ we obtain $\alpha_1\phi'_1 (0-)=0$. Since $\phi'_1 (0-)\neq 0$ we have $\alpha_1=0$ and so $\alpha_2=\alpha_3=0$. Hence $\bold{u}=\bold{0}$.
\item[2)] Let $Z \in (- \frac{2}{\pi} \sum_{j=1}^3  c_j , 0)$. From equation \eqref{trav21} and continuity we have $-c_j^2\phi''_j(0+)=-\sin (\phi_j(0+))=-\sin (\phi_1(0-))=-c_1^2\phi''_1(0-)$. Then from \eqref{spec3} there follows
\begin{equation}
\label{spec4}
-\alpha_1 c_1  \phi''_1(0-)\sum_{j=1}^3  c_j =Z\alpha_1 \phi'_1(0-).
\end{equation}
Suppose $\alpha_1\neq 0$. Then, since $\phi''_1(0-)<0$ and  $\phi'_1(0-)>0$ we obtain a contradiction from \eqref{spec4}. Hence, $\alpha_1=\alpha_2=\alpha_3=0$.
\item[3)] Let $Z\in (-\sum_{j=1}^3  c_j , -\frac{2}{\pi} \sum_{j=1}^3  c_j )$. Suppose $\alpha_1\neq 0$. Then, since 
$$
\phi''_1(0)=4\frac{e^{-\frac{a_1}{ c_1 }}- e^{-\frac{3a_1}{ c_1 }}}{[1+e^{-\frac{2a_1}{ c_1 }}]^2 c_1^2},\quad \phi'_1(0)=\frac{4e^{-\frac{a_1}{ c_1 }}}{[1+e^{-\frac{2a_1}{ c_1 }}]  c_1 },
$$
we obtain from \eqref{trav25} and \eqref{spec4} the relation
\begin{equation}\label{spec5}
(1-y^2) \arctan y = y,\qquad y=e^{-\frac{a_1}{ c_1 }}.
\end{equation}
Since $a_1>0$ we obtain $y\in (0,1)$ and so the function $h(x)=(1-x^2) \arctan x \, -x$ has a zero for $x\in (0,1)$. Since $h(0)=0, h(1)=-1$ and $h'(x)<0$ on $(0,1)$, we obtain a contradiction. Hence, $\alpha_1=\alpha_2=\alpha_3=0$. 
\end{enumerate}

\end{proof}

\begin{proposition}
\label{main3}
Let $Z\in  \big(-\sum_{j=1}^3  c_j , -\frac{2}{\pi} \sum_{j=1}^3  c_j  \big]$. Then $n( \mathcal{L}_Z )=1$.
\end{proposition}

\begin{proof} We will use the extension theory for symmetric operator. In fact, from Proposition \ref{M} in Appendix we obtain that the family $(\mathcal L_Z, D(\mathcal L_Z))$ represents all the self-adjoint extensions of the closed symmetric operator $(\mathcal M_0, D(\mathcal M_0))$ where
\begin{equation}\label{spec6}
\begin{split}
&\mathcal{M}_0=\Big (\Big(-c_j^2\frac{d^2}{dx^2}+ \cos (\phi_j)
\Big)\delta_{j,k} \Big ),\;1\leqq j, k\leqq 3,\\
&D(\mathcal{M}_0)= \Big \{(v_j)_{j=1}^3\in H^2(\mathcal{G}):
v_1(0-)=v_2(0+)=v_3(0+)=0, \, \sum_{j=2}^3 c_j^2v_j'(0+)-c_1^2v_1'(0-)=0 \Big \},
\end{split}
\end{equation}
where $n_{\pm}(\mathcal{M}_0)=1$. Next, we show that  $\mathcal{M}_0\geqq 0$. Let $L_j=-c_j^2\frac{d^2}{dx^2}+ \cos (\phi_j)$, then from \eqref{trav21} we obtain 
\begin{equation}\label{spec7}
L_j\psi=-\frac{1}{\phi'_j} \frac{d}{dx}\Big[c_j^2  (\phi'_j)^2 \frac{d}{dx}\Big(\frac{\psi}{\phi'_j}\Big)\Big].
\end{equation}
We note that always we have $\phi'_j\neq 0$. Thus for $\Psi=(\psi_j)\in D(\mathcal{M}_0)$ we obtain
\begin{equation}\label{spec6}
\begin{split}
\langle \mathcal{M}_0 \Psi,\Psi\rangle&=\int_{-\infty}^0c_1^2(\phi'_1)^2\Big|\frac{d}{dx}\Big(\frac{\psi_1}{\phi'_1}\Big)\Big|^2dx +\sum_{j=2}^3\int_0^{+\infty}c_j^2(\phi'_j)^2\Big|\frac{d}{dx}\Big(\frac{\psi_j}{\phi'_j}\Big)\Big|^2dx\\
&-c_1^2\psi_1(0)\Big[\frac{\psi'_1(0)\phi'_1(0)-\psi_1(0)\phi''_1(0)}{\phi'_1(0)}\Big]+\sum_{j=2}^3c_j^2\psi_j(0)\Big[\frac{\psi'_j(0)\phi'_j(0)-\psi_j(0)\phi''_j(0)}{\phi'_j(0)}\Big]
\end{split}
\end{equation}
The integral terms in \eqref{spec6} are non-negative and equal zero if and only if $\Psi\equiv 0$.  Due to the conditions $\psi_1(0-)=\psi_2(0+)=\psi_3(0+)=0$ non-integral term vanishes, and we get $\mathcal M_0\geqq 0$.

Due to  Proposition \ref{semibounded} (see Appendix \S \ref{secApp}),  we have all the  self-adjoint extensions $\mathcal L_Z$ of $\mathcal M_0$ satisfies $n(\mathcal L_Z)\leqq 1$. Next, for $\Phi=(\phi_1, \phi_2, \phi_3)\in D(\mathcal L_Z)$, it follows from the relations $L_j\phi_j=-\sin (\phi_j)+ \cos (\phi_j)\phi_j$ that
$$
\langle \mathcal L_Z \Phi, \Phi\rangle=\int_{-\infty}^0[-\sin (\phi_1)+ \cos (\phi_1)\phi_1]\phi_1dx +\sum_{j=2}^3\int_0^{+\infty}[-\sin (\phi_j)+ \cos (\phi_j)\phi_j]\phi_jdx<0,
$$
because of $0<\phi_j(x)\leqq \pi$ for every $Z\in \big(-\sum_{j=1}^3  c_j , -\frac{2}{\pi} \sum_{j=1}^3  c_j  \big]$ and $\theta  \cos  \theta\leqq \sin  \theta$ for all $\theta\in [0, \pi]$.  Then from minimax principle we  arrive at $n(\mathcal L_Z)=1$. This finishes the proof.
\end{proof}

\begin{remark}\label{noway}For the  case $Z\in \big(-\frac{2}{\pi} \sum_{j=1}^3  c_j ,0 \big)$ in Proposition \ref{main3}, it is no clear for us if the extension theory approach can give us the exact value of the Morse-index of $\mathcal L_Z$ for every $Z$. Indeed, the  nonnegative property for $\mathcal F_0$ is still right from \eqref{spec6}, but the quadratic form $\langle \mathcal L_Z \Phi, \Phi\rangle$ may have an undefined sign because of $\phi_i\in (0, \eta_0)$, $\eta_0=4  \arctan (e^ {-\frac{1}{ c_1 }a_1})>\pi$. We note here that the inequality $\theta  \cos  \theta\leqq \sin  \theta$ still is true for all $\theta\in [0, \theta_0]$ where $\theta_0\approx 4.4934$ is the unique zero for $h(x)=\tan x-x$ in the interval $(\pi, 2\pi)$. Thus, it is not difficult to see that for specific intervals of $a_1<0$ we have either $\eta_0<\theta_0$ or $\eta_0>\theta_0$. As we will see in the following proposition the property $n(\mathcal L_Z)=1$ is still true.

\end{remark}

\begin{proposition}
\label{main4}
Let $Z\in \big(-\frac{2}{\pi} \sum_{j=1}^3  c_j ,0\big)$. Then $n( \mathcal{L}_Z )=1$.
\end{proposition}

\begin{proof} We will use analytic perturbation theory. Initially, from  Proposition \ref{main3} we have for $Z^*=-\frac{2}{\pi} \sum_{j=1}^3  c_j $ that  $n( \mathcal{L}_{Z^*} )=1$. Now, from \eqref{trav25}-\eqref{trav26} we have the continuous  mapping function $Z\in \big(-\sum_{j=1}^3  c_j ,0\big)\to a_1(Z)$ such that 
\begin{equation}
a_1(Z)=\begin{cases}
\begin{aligned}
&<0,\quad \text{for}\;\;Z^*<Z<0,\\
&=0,\quad \text{for}\;\;Z=Z^*,\\
&>0,\quad \text{for}\;\;-\sum\limits_{j=1}^3  c_j <Z<Z^*.
\end{aligned}
\end{cases}
\end{equation}
Thus, by denoting the  stationary profiles in \eqref{trav24} as a function of $Z$, $\Phi_{a_1(Z)}=(\phi_{j, a_1(Z)})$ we obtain the convergence $\Phi_{a_1(Z)}\to \Phi_0$ as $Z\to Z^*$ in $H^1 (\mathcal G)$. Here $\Phi_{0}=(\phi_{j, 0})$.

Next, we obtain that  $\mathcal L_Z$ converges to $\mathcal L_{Z^*}$ as $Z \to Z^*$ in the generalized sense. Indeed,  denoting 
$$
W_Z=\Big( \cos (\phi_{j, a_1(Z)})\delta_{j,k}\Big)
$$
we obtain 
\begin{equation*}
 \begin{split}
 \widehat{\delta}(\mathcal L_Z, \mathcal L_{Z^*} )&=\widehat{\delta}(\mathcal L_{Z^*} + (W_Z-W_{Z^*}), \mathcal L_{Z^*})\leqq \|W_Z-W_{Z^*}\|_{L^2 (\mathcal G)}\to 0,\qquad\text{as}\;\;Z\to Z^*,
\end{split}
\end{equation*}
where $\widehat{\delta}$ is  the gap metric  (see \cite[Chapter IV]{kato}). 

Now, from Proposition \ref{main2} and from  Morse-index for $\mathcal L_{Z^*}$ being one, we can separate the spectrum $\sigma(\mathcal L_{Z^*})$ of $\mathcal L_{Z^*}$   into two parts $\sigma_0=\{\lambda^*\}$ ($\lambda^*<0$) and $\sigma_1$ by a closed curve  $\Gamma$ belongs to the resolvent set of $\mathcal L_{Z^*}$ with $0\in \Gamma$ and   such that $\sigma_0$ belongs to the inner domain of $\Gamma$ and $\sigma_1$ to the outer domain of $\Gamma$. Moreover, $\sigma_1\subset [\theta_{Z^*}, +\infty)$ with $\theta_{Z^*}=\inf\{\lambda: \lambda\in \sigma(\mathcal L_{Z^*}),\; \lambda>0\}>0$. Then, by  \cite[Theorem 3.16, Chapter IV]{kato}, we have  $\Gamma\subset \rho(\mathcal L_{Z})$ for $Z\in [Z^*-\delta_1, Z^*+\delta_1]$ and $\delta_1>0$ small enough. Moreover, $\sigma(\mathcal L_{Z})$ is likewise separated by $\Gamma$ into two parts so that the part of $\sigma (\mathcal L_{Z})$ inside $\Gamma$ (negative eigenvalues) will consist exactly of a unique negative eigenvalue with total multiplicity (algebraic) one. Therefore, $n(\mathcal L_{Z})=1$ for $Z\in [Z^*-\delta_1, Z^*+\delta_1]$.

Next, we use a classical continuation argument based on the Riesz-projection for showing that $n( \mathcal{L}_Z )=1$ for all $Z\in  (Z^*,0)$. Indeed, define $\omega$ by
$$
\omega= \sup \left\{\eta: \eta\in (Z^*, 0) \; \text{s.t.} \; n(\mathcal{L}_Z)=1\;\;\text{for all}\;\;Z\in (Z^*, \eta)\right\}.
$$
Analysis above implies that $\omega$ is well defined, and $\omega\in (Z^*, 0)$. We claim that $\omega=0$.  Suppose that $\omega<0$.  Let $N=n(\mathcal{L}_\omega)$, and $\Gamma$ be  a closed curve such that $0\in  \Gamma\subset \rho(\mathcal L_\omega)$, and all the negative eigenvalues  of $\mathcal L_\omega$ belong to the inner domain of $\Gamma$. Next,  using that as a function of $Z$, $(\mathcal{L}_Z)$ is a real-analytic family of self-adjoint operators of type $(B)$ in the sense of Kato  (see \cite{kato})  we deduce that there is  $\epsilon>0$ such that for $Z\in [\omega-\epsilon, \omega+\epsilon]$ we have $\Gamma\subset \rho(\mathcal L_{Z})$,  and  the mapping $Z\to (\mathcal L_{Z}-\xi)^{-1}$ is analytic for  $\xi\in \Gamma$. Therefore, the existence of an analytic family of Riesz-projections $Z\to P(Z)$ given by
$$
P(Z)=-\frac{1}{2\pi i}\ointctrclockwise_{\Gamma} (\mathcal L_Z-\xi)^{-1}d\xi 
$$
implies that $\dim(\mathrm{range} \, P(Z))=\dim(\mathrm{range} \, P(\omega))=N$ for all $Z\in [\omega-\epsilon, \omega+\epsilon]$. Further, by definition of $\omega$, there is $\eta_0\in (\omega-\epsilon,\omega)$ and $\mathcal{L}_Z$ has $n(\mathcal{L}_Z)=1$ for all $Z\in (Z^*,\eta_0)$. Therefore, $N=1$ and $\mathcal L_{\omega+\epsilon}$ has exactly one negative eigenvalue, hence $\mathcal L_{Z}$ has $n(\mathcal{L}_Z)=1$  for $Z\in (Z^*, \omega+\epsilon)$, which contradicts with the definition of $\omega$. Thus, $\omega= 0$. This finishes the proof.

\end{proof}

\begin{remark} It is immediate from the proof of Proposition \ref{main4} that we recover the result of $n( \mathcal{L}_Z )=1$ for $Z\in  (-\sum_{j=1}^3  c_j , -\frac{2}{\pi} \sum_{j=1}^3  c_j  )$ in Proposition \ref{main3}. Moreover, we do not know another approach (other than extension theory) for showing that $n( \mathcal{L}_{Z^*} )=1$.
\end{remark}

\begin{proof}[Proof of Theorem \ref{1main}] 
The spectral (linear) instability of the stationary profiles $\Phi_Z$ follows from a direct application of Propositions \ref{main2}, \ref{main3}, \ref{main4}, and Theorem \ref{crit}. This finishes the proof.
\end{proof}

\begin{remark}
\label{remnolineal}
Since the mapping data-solution for the sine-Gordon model on $\mathcal E(\mathcal Y)\times L^2(\mathcal Y)$ is at least of class $C^2$ (in fact, it is smooth), by Theorem \ref{well0} and from the approach by Henry {\it et al.} \cite{HPW82} (see also Angulo and Natali \cite{AngNat16}, Angulo \emph{et al.} \cite{ALN08}) it follows that the linear instability property of $\Phi_Z$ is actually of nonlinear type in the $H^1(\mathcal Y)\times L^2(\mathcal Y)$-norm by the flow of the sine-Gordon model. In other words, the spectral (linear) stability result presented here implies the orbital (nonlinear) stability of the static solutions of kink-type. The well-posedness of the Cauchy problem is an essential ingredient to reach this conclusion. The reader is referred to \cite{AngNat16,ALN08,HPW82} for further information.
\end{remark}

\section{Instability of stationary solutions of kink and anti-kink type  for the sine-Gordon equation with $\delta$-interaction}
\label{aksect}
In this section apply the linear instability criterion (Theorem \ref{crit} above) to the case of stationary solutions of kink and anti-kink type of the form \eqref{trav-akink}, determined by a $\delta$-interaction type at the vertex  $\nu=0$. For concreteness and without loss of generality, we consider $c_j=1$, $j=1,2,3$, in \eqref{sg2}-\eqref{trav21}. Hence, the structure of such stationary wave solutions, $\Phi =(\phi_j)_{j=1}^3$, are given in this case by
\begin{equation}
\label{ak1}
\begin{split}
&\phi_1(x)=4  \arctan \Big(e^{-(x-a_1)}\Big),\;\;\;\;\;x<0, \;\;\; \lim_{x\to -\infty}\phi_1(x)=2\pi\\
& \phi_i(x)=4  \arctan \Big(e^{-(x-a_i)}\Big), \;\;\;\;\; x>0,\;\;\;  \lim_{x\to +\infty}\phi_i(x)=0,\;\;i=2,3,
\end{split}
\end{equation}
and the conditions $ \phi_1(0-)=\phi_2(0+)= \phi_3(0+)$ and  $\sum_{i=2}^3  \phi'_i(0+) - \phi'_1(0-)=Z  \phi_1(0-)$. Thus we obtain immediately  the ``same shift property", $a_1=a_2=a_3$, and from the equality
\begin{equation}\label{ak2}
-\frac{e^{a_1}}{1+e^{2a_1}}=Z \arctan \Big(e^{a_1}\Big),
\end{equation}
the condition $Z\in (-1,0)$.  Moreover, we have the following specific behavior of the profiles $ \phi_i$:
\begin{enumerate}
\item[1)] for  $Z\in (-1 , -\frac{2}{\pi} )$ we obtain $a_1<0$; therefore $ \phi_i'' >0$ for every $i=2,3$, and $\phi_1''(a_1)=0$. Thus, the profile of $(\phi_j)_{j=1}^3$  should look similar to Figure \ref{figAKI} below, namely, two left-translated anti-kink on all the line. Moreover, $\phi_i(0)\in (\eta, \pi)$, $i=1,2, 3$, $\eta>0$,
\item[2)] for   $Z\in (-\frac{2}{\pi}, 0)$ we obtain $a_1>0$; therefore $\phi_i''(a_1) =0$,  $i=2, 3$. 
Thus, the profile of $(\phi_j)_{j=1}^3$  should look similar to Figure \ref{figAKII} below, namely, two right-translated anti-kink on all the line. Moreover, $\phi_i(0)\in (\pi, 2\pi)$, $i=1,2, 3$,
\item[3)]  the case  $Z=-\frac{2}{\pi} $ implies $a_1=a_2=a_3=0$; therefore,  $\phi_i(0) =\pi$ and   $\phi_i''(0) =0$, $i=1,2, 3$. In this case, we have two-classical anti-kink profile around the vertex $\nu=0$ (see Figure \ref{figAKIII} below). 
\end{enumerate}

\begin{figure}[t]
\begin{center}
\subfigure[$Z\in (-1,-2/\pi)$]{\label{figAKI}\includegraphics[scale=.42, clip=true]{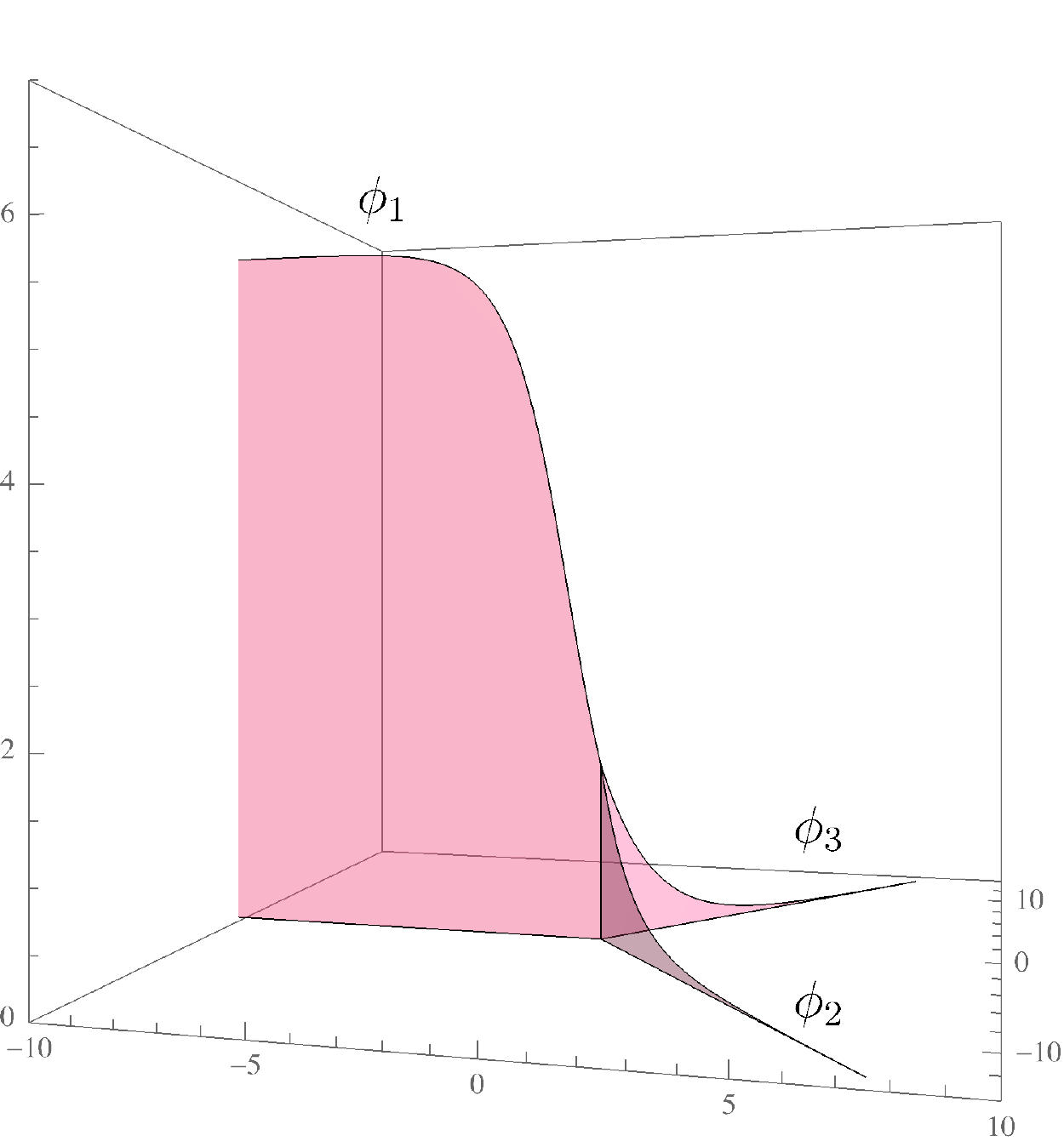}}
\subfigure[$Z\in (-2/\pi, 0)$]{\label{figAKII}\includegraphics[scale=.42, clip=true]{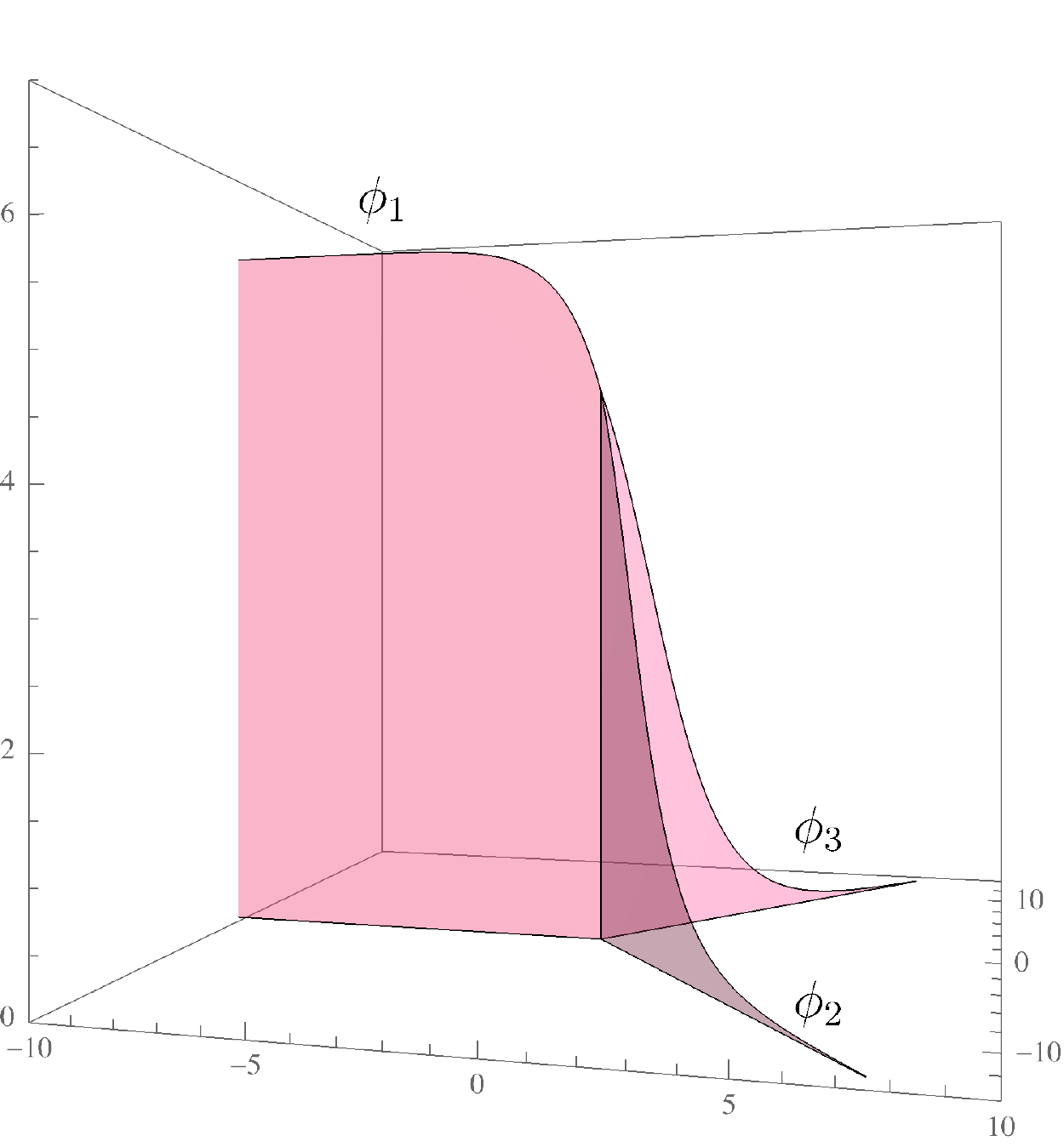}}
\subfigure[$Z=-2/\pi$]{\label{figAKIII}\includegraphics[scale=.42, clip=true]{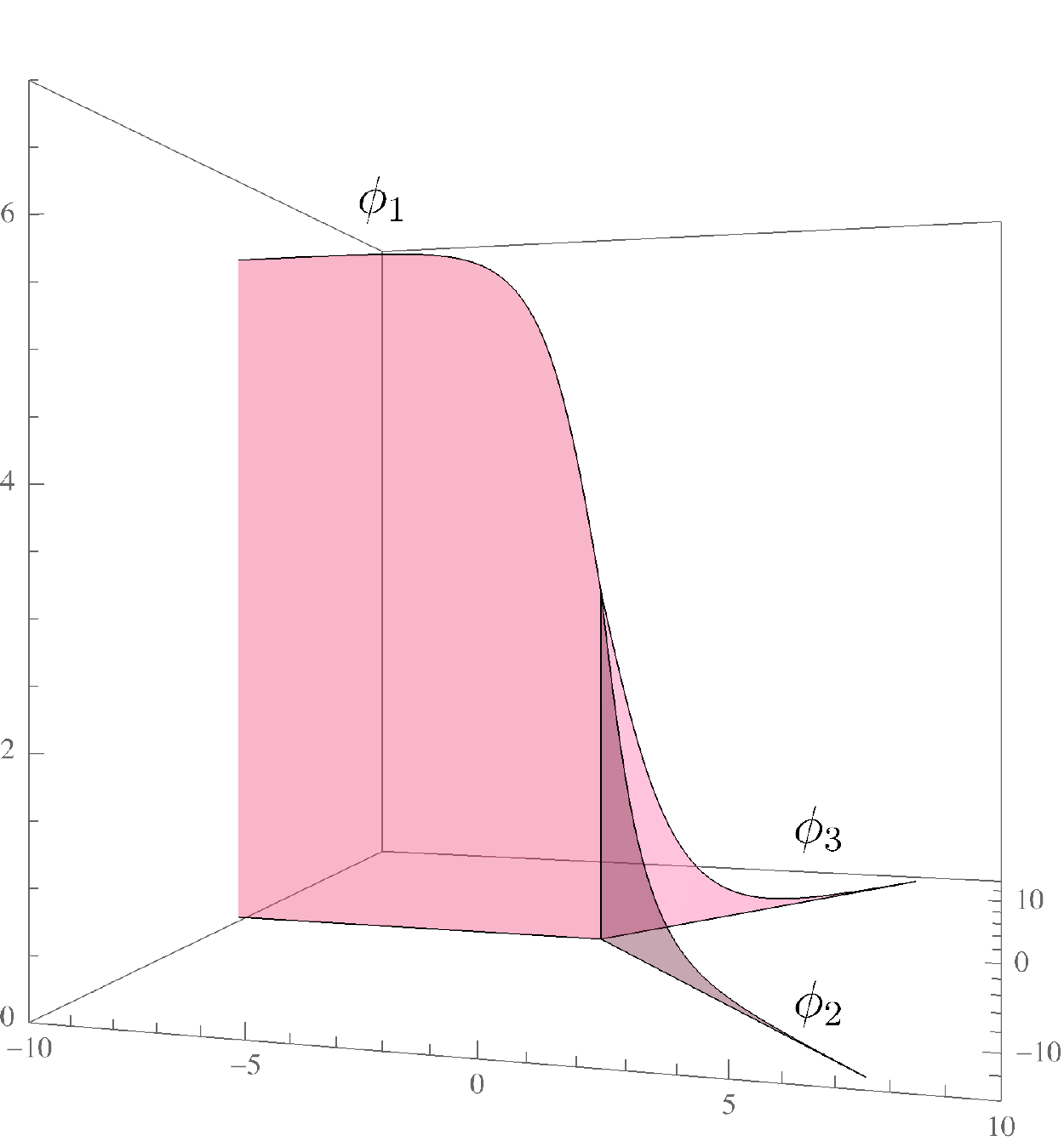}}
\end{center}
\caption{\small{Plots of stationary solutions of anti-kink/kink type of the form \eqref{ak1}, in the case where $c_j = 1$ for all $j=1,2,3$, for different values of $Z\in (-1,0)$. Panel (a) shows the stationary profile solutions (left-translated anti-kink configuration) for the value $Z = -0.9 \in (-1,-2/\pi)$. Panel (b) shows the profile (right-translated anti-kink) for the value $Z = -0.25 \in (-2/\pi,0)$. Panel (c) shows the profile solution for the parameter value $Z = -2/\pi$ (color online).}}\label{figAntikink}
\end{figure}

The main stability result for the stationary profiles
$\Pi_Z=(\phi_1, \phi_2, \phi_3, 0,0,0)$ with $\phi_j$ defined in \eqref{ak1} is the following.
\begin{theorem}
\label{akmain}
Let $Z\in  (-1 , 0)$ and the smooth family of stationary profiles $Z\to \Pi_Z$ defined in \eqref{ak1}. Then $\Pi_Z$ is spectrally unstable for the sine-Gordon model \eqref{stat1}.
\end{theorem}
It is to be noticed that the kink/anti-kink stationary profiles \eqref{ak1} do not belong to the energy space $H^2(\mathcal{Y})$. Therefore, in order to be able to apply the linear instability criterion of section \S \ref{seccriterium}, we need to verify hypotheses $(S_1)-(S_3)$ with respect to the flow generated by finite energy perturbations of these static solutions. A similar analysis on the well-posedness of perturbed solutions around unbounded periodic wave solutions of subluminal librational type for the sine-Gordon equation can be found in \cite{AnPl16}. In the sequel we establish the appropriate framework for the application of Theorem \ref{crit}.

\subsection{Functional space for stability properties of the kink/anti-kink profile}
\label{akfunspa}
The natural framework space for studying  stability properties of the kink/anti-kink profile $\Phi=(\phi_j)$ for the sine-Gordon model is $\mathcal X(\mathcal Y)= H^1_{\mathrm{loc}}(-\infty, 0)\bigoplus H^1(0, \infty)\bigoplus H^1(0, \infty)$. Thus we say that a flow $t\to (u(t), v(t))\in \mathcal X(\mathcal Y) \times L^2(\mathcal Y)$ is called  a \emph{perturbed solution} for the anti-kink profile $\Phi\in \mathcal X(\mathcal Y)$ if for $(P(t), Q(t))\equiv (u(t)-\Phi,v(t))$ we have that $ (P(t), Q(t))\in H^1(\mathcal Y)\times L^2(\mathcal Y)$ and  
$\bold z=(P,Q)^\top$ satisfies the following vectorial perturbed sine-Gordon model 
 \begin{equation}\label{akmodel}
 \begin{cases}
  \bold z_t=JE\bold z +F_1(\bold z)\\
 P(0)=u(0)-\Phi \in H^1(\mathcal Y),\\
 Q(0)=v(0)\in  L^2(\mathcal Y),
  \end{cases}
 \end{equation} 
  where for $P=(p_1, p_2, p_3)$ we have
   \begin{equation}
  F_1(\bold z)=\left(\begin{array}{cc}  0\\  0   \\  0   \\  \sin(\phi_1)-\sin (p_1+\phi_1)   \\  \sin(\phi_2)-\sin (p_2+\phi_2)   \\  \sin(\phi_3)-\sin (p_3+\phi_3)  \end{array}\right).
  \end{equation} 
Then, by studying stability linear properties of the anti-kink profile $\Pi_Z$ by the sine-Gordon model on $\mathcal X(\mathcal Y) \times L^2(\mathcal Y)$ is reduced to study stability properties of the trivial solution $(P,Q)=(0,0)$ for the  linearized  model \eqref{akmodel} around $(P,Q)=(0,0)$. Thus, via Taylor's Theorem  we obtain the  linearized system in \eqref{stat4} but with the Schr\"odinger diagonal operator $\mathcal L$ in \eqref{stat5} now determined by the anti-kink profile $\Phi=(\phi_j)$. We will denote this operator by $\mathcal L_{ak, Z}$, with domain $D(\mathcal L_{ak, Z})$ determined by the Kirchhoff's boundary conditions defined in \eqref{2trav23} with $c_j=1$. In this way, we can apply {\it{ipsis litteris}} the  semi-group theory results in section \S \ref{seclocalWP} to the operator $JE$ and to the local well-posedness problem in $\mathcal E(\mathcal Y)\times L^2(\mathcal Y)$ for the vectorial perturbed sine-Gordon model \eqref{akmodel}. We note that the anti-kink profile $\Phi\in \mathcal X(\mathcal Y)$ but $\Phi'\in H^2(\mathcal Y)$. 

\subsection{Spectral study}
\label{akspectral}
Next we determine  the assumptions $(S_1)-(S_3)$ required by Theorem \ref{crit} for obtaining that the eigenvalue problem $J\mathcal E\Psi=\lambda \Psi$ has a non-trivial solution with $\RE \lambda >0$. Assumptions $(S_1)-(S_2)$ are verified similarly as in section \S \ref{seccriterium}. The spectral conditions required by assumption $(S_3)$ are given in the following propositions.
\begin{proposition}\label{akmain2}
Let $Z\in  (-1, 0)$. Then $ \ker ( \mathcal{L}_{ak,Z} )=\{\mathbf{0}\}$ and $\sigma_{\mathrm{ess}}(  \mathcal{L}_{ak,Z})=[1, +\infty)$.
\end{proposition}
\begin{proof} 
The proof is {\it{ipsis litteris}} as in the case of Proposition \ref{main2}. Indeed, since $\Phi'\in H^2(\mathcal Y)$ the Sturm-Liouville theory on half-lines  can be applied and so for  $\bold{u}=(u_1, u_2, u_3)\in D(\mathcal{L}_{ak,Z} )$ and $\mathcal{L}_{ak,Z} \bold{u}=\bold{0}$, it implies $u_j=\alpha \phi'_j$ for some $\alpha\in \mathbb R$ (here we also use that $\phi_j'(0+)=\phi'_1(0-)$ for $j=2,3$). Moreover, the Kirchhoff's condition implies also $\alpha\phi ''_1(0-)= \alpha Z\phi '_1(0-)$. Thus by considering the cases $Z\in (-1, -\frac{2}{\pi})$, $Z\in (-\frac{2}{\pi}, 0)$ and $Z=-\frac{2}{\pi}$ we obtain $\alpha=0$. 
Then by Weyl's theorem and the calculations $\lim_{x\to -\infty}\cos(\phi_1(x))=1$ and $\lim_{x\to +\infty}\cos(\phi_j(x))=1$, $j=2,2$, we conclude that $\sigma_{\mathrm{ess}}(  \mathcal{L}_{ak,Z})=[1, +\infty)$.
\end{proof}

\begin{proposition}\label{akmain3}
Let $Z\in  \big[ -\frac{2}{\pi},0)$. Then $n( \mathcal{L}_{ak,Z} )=1$.
\end{proposition}
\begin{proof} The idea of the proof is to use the extension theory for symmetric operators as in Proposition \ref{main3}. Thus,, since $\phi'_j\neq 0$ for all $j$ we have $n( \mathcal{L}_{ak,Z} )\leqq 1$. In order to show that $n( \mathcal{L}_{ak,Z} )\geqq 1$ we consider the following quadratic form $\mathcal Q$ associated to $(\mathcal{L}_{ak,Z}, D(\mathcal{L}_{ak,Z}))$ on the space $\mathcal E(\mathcal Y)=  \{(u_j)_{j=1}^3\in H^1(\mathcal{Y}):
u_1(0-)=u_2(0+)=u_3(0+) \}$, defined as
  \begin{equation}\label{akqua}
\mathcal Q(\bold u)=\int_{-\infty}^0 (u'_1(x))^2 + \cos(\phi_1(x)) u^2_1(x)dx +\sum_{j=2}^3 \int_0^{+\infty} (u'_j(x))^2 + \cos(\phi_j(x)) u^2_j(x)dx +Z[u_1(0)]^2,
\end{equation}
for each $\bold u=(u_j)\in \mathcal E(\mathcal Y)$.
Next, since $\Phi'=(\phi_j')\in \mathcal E(\mathcal Y)$ we get via integration by parts and from the equality $\phi_j''=\sin(\phi_j)$ the relation
 \begin{equation}\label{akqua2}
\mathcal Q(\Phi')= Z[\phi'_1(0)]^2 - \phi'_1(0) \phi''_1(0).
\end{equation}
Lastly, since for all $Z\in  \big[ -\frac{2}{\pi},0)$ we have $\phi'_1(0) \phi''_1(0)\geqq 0$ we get $\mathcal Q(\Phi')<0$ and so via the minimax principle $n( \mathcal{L}_{ak,Z} )\geqq 1$. This finishes the proof.
\end{proof}

Now, similarly as in the stability study in section \S \ref{secInst}, we have that for $Z\in (-1, -\frac{2}{\pi})$ is not clear if the extension theory approach can give us the relation $n( \mathcal{L}_{ak,Z} )\geqq 1$. Thus we will use again analytic perturbation theory.
\begin{proposition}
\label{akmain4}
Let $Z\in \big(-1, -\frac{2}{\pi}\big)$. Then $n( \mathcal{L}_{ak,Z} )=1$.
\end{proposition}
\begin{proof} The proof is {\it{ipsis litteris}} as in the case of Proposition \ref{main4}. Indeed, we only need to note that in this case we have the continuous mapping function $Z\in \big(-1,0\big)\to a_1(Z)$ such that 
\begin{equation}
a_1(Z)=\begin{cases}
\begin{aligned}
&>0,\quad \text{for}\;\;-\frac{2}{\pi}<Z<0,\\
&=0,\quad \text{for}\;\;Z=-\frac{2}{\pi},\\
&<0,\quad \text{for}\;\;-1<Z<-\frac{2}{\pi},
\end{aligned}
\end{cases}
\end{equation}
and  so we obtain that 
$$
\lim_{Z\to -\frac{2}{\pi}}\|\phi_{1, a_1(Z)}-\phi_{1, a_1(-\frac{2}{\pi})}\|_{H^1(-\infty, 0)}=0.
$$
which implies immediately that $\mathcal L_{ak,Z}$ converges to $\mathcal L_{ak, -\frac{2}{\pi}}$ as $Z\to -\frac{2}{\pi}$ in the generalized sense. This finishes the proof.
\end{proof}
\begin{proof}[Proof of Theorem \ref{akmain}] 
The spectral (linear) instability of the stationary profiles $\Pi_Z$ follows from Propositions \ref{akmain2}, \ref{akmain3}, \ref{akmain4}, and a direct application of Theorem \ref{crit} applied to the linearized vectorial perturbed sine-Gordon model \eqref{akmodel} around the trivial solution $(0,0)$. This finishes the proof.
\end{proof}

\section{Discussion and open problems}
\label{secconclu}
In this paper, we have established the linear instability of static wave solutions to the sine-Gordon equation posed on a $\mathcal{Y}$-junction with boundary conditions of $\delta$-type at the vertex (equations \eqref{bcI}). These conditions are characterized by a parameter $Z \in \R$ and represent continuity of the wave functions at the vertex and a flux balance of intensity $Z$. The static wave solutions under consideration are of kink \eqref{trav22}, or kink/anti-kink \eqref{trav-akink} type. To that end, we have established a general linear instability criterion which essentially provides the sufficient conditions for the linearized operator around any static solution, $\mathcal{L}_Z$, to have a pair of positive/negative real eigenvalues. Consequently, the linear stability analysis depends upon of the spectral study of the linearized operator and of its Morse index. The extension theory of symmetric operators, Sturm-Liouville oscillation results and analytic perturbation theory of operators are fundamental ingredients in this endeavor. The linear instability criterion introduced in this work has prospects for the study of other types of stationary wave solutions (such as breathers or anti-kinks), and/or of other types of interactions at the vertex. For example and up to our knowledge, there are no rigorous results on the stability of solutions of breather type in the literature (see the recent work \cite{Susa19} for a numerical study of this problem). The stability properties of kink and kink/anti-kink soliton profile solutions for the sine-Gordon equation on a $\mathcal{Y}$-junction, but with boundary conditions of $\delta'$-type, will be addressed in a companion paper \cite{AnPl}.

\section*{Acknowledgements}
The authors thank two anonymous referees for their insightful comments and suggestions which improved the quality of the paper. The authors also thank J.-G. Caputo for physical interpretations on the boundary condition \eqref{Kirchhoffbc}. J. Angulo  was supported in part by CNPq/Brazil Grant and FAPERJ/Brazil program PRONEX-E - 26/010.001258/2016. The work of R. G. Plaza was partially supported by DGAPA-UNAM, program PAPIIT, grant IN-100318.

\appendix
\section{Appendix}
\label{secApp}

For convenience of the reader, and because of the use of non-standard techniques along the manuscript, in this section we formulate the extension theory of symmetric operators suitable for our needs (see \cite{Nai67,Nai68} for further information). The following result is classical and can be found in \cite{RS2}, p. 138.
 
\begin{theorem}[von-Neumann decomposition]\label{d5} 
Let $A$ be a closed, symmetric operator, then
\begin{equation}\label{d6}
D(A^*)=D(A)\oplus\mathcal N_{-i} \oplus\mathcal N_{+i}.
\end{equation}
with $\mathcal N_{\pm i}= \ker (A^*\mp iI)$. Therefore, for $u\in D(A^*)$ and $u=x+y+z\in D(A)\oplus\mathcal N_{-i} \oplus\mathcal N_{+i}$,
\begin{equation}\label{d6a}
A^*u=Ax+(-i)y+iz.
\end{equation}
\end{theorem}

\begin{remark} The direct sum in (\ref{d6}) is not necessarily orthogonal.
\end{remark}

The following propositions provide a strategy for estimating the Morse-index of the self-adjoint extensions (see Reed and Simon, vol. 2, \cite{RS2}, chapter X).

\begin{proposition}\label{semibounded}
Let $A$  be a densely defined lower semi-bounded symmetric operator (that is, $A\geq mI$)  with finite deficiency indices, $n_{\pm}(A)=k<\infty$,  in the Hilbert space $\mathcal{H}$, and let $\widehat{A}$ be a self-adjoint extension of $A$.  Then the spectrum of $\widehat{A}$  in $(-\infty, m)$ is discrete and  consists of, at most, $k$  eigenvalues counting multiplicities.
\end{proposition}

\begin{proposition}\label{11}
	Let $A$ be a densely defined, closed, symmetric operator in some Hilbert space $H$ with deficiency indices equal  $n_{\pm}(A)=1$. All self-adjoint extensions $A_\theta$ of $A$ may be parametrized by a real parameter $\theta\in [0,2\pi)$ where
	\begin{equation*}
	\begin{split}
	D(A_\theta)&=\{x+c\phi_+ + \zeta e^{i\theta}\phi_{-}: x\in D(A), \zeta \in \mathbb C\},\\
	A_\theta (x + \zeta \phi_+ + \zeta e^{i\theta}\phi_{-})&= Ax+i \zeta \phi_+ - i \zeta e^{i\theta}\phi_{-},
	\end{split}
	\end{equation*}
	with $A^*\phi_{\pm}=\pm i \phi_{\pm}$, and $\|\phi_+\|=\|\phi_-\|$.
\end{proposition}

Next Proposition can be found in Naimark \cite{Nai68} (see Theorem 9, p. 38).

\begin{proposition}
\label{esse}
All self-adjoint extensions of a closed, symmetric operator
which has equal and finite deficiency indices have one and the
same continuous spectrum.
\end{proposition}

The following proposition is the main result of this Appendix and characterizes all self-adjoint extensions of the symmetric operator under consideration. It plays a key role in the proof of Proposition  \ref{main3}.

\begin{proposition}\label{M}
Consider the closed symmetric operator densely defined on $L^2(\mathcal Y)$, $( \mathcal M, D( \mathcal M))$, by 
\begin{equation}\label{M1}
\begin{split}
&\mathcal{M}=\Big (\Big(-c_j^2\frac{d^2}{dx^2}\Big)\delta_{j,k} \Big ),\;\l1\leqq j, k\leqq 3,\\
&D(\mathcal{M})= \Big \{(v_j)_{j=1}^3\in H^2(\mathcal{Y}):
v_1(0-)=v_2(0+)=v_3(0+)=0,\;\; \sum\limits_{j=2}^3 c_j^2v_j'(0+)-c_1^2v_1'(0-)=0 \Big \},
\end{split}
\end{equation}
with $\delta_{j,k}$ being the Kronecker symbol. Then, the deficiency indices   are $n_{\pm}( \mathcal M)=1$. Therefore, we have that all the self-adjoint extensions of $( \mathcal M, D( \mathcal M))$, namely, $(\mathcal J_Z, D(\mathcal J_Z))$,  $Z\in \mathbb R$, are defined by  $\mathcal J_Z\equiv \mathcal M$ and $D(\mathcal J_Z)$  by \eqref{2trav23}.
\end{proposition}

\begin{proof}
We show initially  that the adjoint operator $( \mathcal M^*, D( \mathcal M^*))$ of $( \mathcal M, D( \mathcal M))$   is given by 
  \begin{equation}\label{F*2}
 \mathcal M^*=\mathcal M, \quad D( \mathcal M^*)=\{u\in H^2(\mathcal Y) : u_1(0-)=u_2(0+)=u_3(0+)\}.
 \end{equation}
Indeed, formally for $\bold u=(u_1, u_2, u_3), \bold v=(v_1, v_2, v_3)\in H^2(\mathcal Y)$ we have 
\begin{equation}\label{relaself}
\begin{split}
 \langle\mathcal M\bold v, \bold u\rangle
 &=- c_1^2v_1' (0-)u_1(0-)+ c_1^2v_1 (0-)u'_1(0-) + \sum_{j=2}^ 3c_j^2v'_{j}(0+)u_j(0+) 
 - \sum_{j=2}^ 3c_j^ 2v_{j}(0+)u'_j(0+)\\
 &  +\langle \bold v,\mathcal M \bold u\rangle.
 \end{split}
 \end{equation}
 From \eqref{relaself}, we obtain immediately for  $\bold u=(u_1, u_2, u_3), \bold v=(v_1, v_2, v_3)\in  D(\mathcal M)$ the symmetric property of $\mathcal M$. Next, 
 we denote by $D^*=\{u\in H^2(\mathcal G) : u_1(0-)=u_2(0+)=u_3(0+)\}$ and we will show $D^*=D( \mathcal M^*)$. Indeed, from \eqref{relaself}  we obtain for $\bold u\in D^*$ and $\bold v\in D( \mathcal M)$ that  $\langle\mathcal M\bold v, \bold u\rangle= \langle \bold v, \mathcal M\bold u\rangle$ and so $\bold u\in D (\mathcal M^*)$ with $\mathcal M^*\bold u=  \mathcal M\bold u$. Let us show the inverse inclusion $D^*\supseteq
D(\mathcal M^*)$. Take $\bold u=(u_1, u_2, u_3)\in D(\mathcal M^*)$, then
for any $\bold v=(v_1, v_2, v_3)\in D(\mathcal M)$ we have from \eqref{relaself}
\begin{align}\label{relaself3}
 \langle\mathcal M \bold v, \bold u\rangle=- c_1^2v_1' (0-)u_1(0-)+ \sum_{j=2}^ 3c_j^2v'_{j}(0+)u_j(0+) 
+\langle \bold v,\mathcal M \bold u\rangle=\langle \bold v,\mathcal M^* \bold u\rangle=\langle \bold v,\mathcal M \bold u\rangle.
 \end{align}
Thus, we arrive for any $\bold v\in D(\mathcal M)$ at the equality 
\begin{equation}\label{adjoint}
 \sum_{j=2}^ 3c_j^2v'_{j}(0+)u_j(0+) - c_1^2v_1' (0-)u_1(0-)=0
 \end{equation} 
Next, it consider $\bold v=(v_1, 0, v_3)\in D(\mathcal M)$ then $c_1^2 v_1'(0-)=c_3^2 v_3'(0+)$. Then, from \eqref{adjoint} we obtain
$$
[u_3(0+)-u_1(0-)]c_1^ 2 v_1'(0-)=0.
$$
So, by choosing $v_1'(0-)\neq 0$ we obtain $u_1(0-)=u_3(0+)$. Then \eqref{adjoint}  is reduced to $ [u_2(0+)-u_1(0-)]c_2^ 2v_2'(0+)=0$. Therefore, by choosing  $\bold v=(v_1, v_2, v_3)\in D(\mathcal M)$ with $v_2'(0+)\neq 0$ we conclude that $\bold u\in D^*$. This shows relations in \eqref{F*2}.

From \eqref{F*2} we obtain that the deficiency indices  for $( \mathcal M, D( \mathcal M))$ are $n_{\pm}( \mathcal M)=1$. Indeed,  $ \ker (\mathcal M^*\pm iI)=[\Psi_{\pm}]$ with $\Psi_{\pm}$ defined by
\begin{equation}
\Psi_{\pm}=
\Big(\underset{x<0}{e^{\frac{ ik_{\mp}}{ c_1 }x}}, \underset{x>0}{e^{\frac{- ik_{\mp}}{ c_2 }x}}, \underset{x>0}{e^{\frac{ -ik_{\mp}}{ c_3 }x}}\Big),
\end{equation}
$k^2_{\mp}=\mp i$, $\IM (k_{-})<0$ and $\IM (k_{+})<0$. Moreover, $\|\Psi_{-}\|_{L^ 2(\mathcal Y)}=\|\Psi_{+}\|_{L^ 2(\mathcal Y)}$.
 
 Next, let us show that the domain of any self-adjoint extension
$\widehat{\mathcal M}$ of the operator $(\mathcal M, D(\mathcal M))$ in \eqref{M1} (and acting on complex-valued functions) is given by $D(\widehat{\mathcal M})=D({\mathcal J_Z})$ in \eqref{trav23}. Indeed, we recall from extension theory for symmetric operator that $D(\widehat{\mathcal M})$ is a restriction of $D({\mathcal M}^*)$ (von-Neumann decomposition above), so $ D(\widehat{\mathcal M})\subset D({\mathcal M}^*)$ (continuity at the vertex $\nu=0$). Next,  due to Proposition \ref{11} follows
$$
D(\widehat{\mathcal M})=\left\{\bold u\in H^2(\mathcal Y), \, \bold u= u_0+ \zeta \Psi_{-} + \zeta e^{i\theta}\Psi_{+}:\,  u_0\in D(\mathcal M), \zeta \in\mathbb{C},\theta\in[0,2\pi)\right\},
$$
Thus, it is easily seen that for $\bold u=(u_1,u_2, u_3)\in D(\widehat{\mathcal M})$, we
have
\begin{align}\label{Zcondi}
\sum\limits_{j=2}^3 c_j^2u_j'(0+)-c_1^2u_1'(0-)= - \zeta \sum\limits_{j=1}^3  c_j  \Big(e^ {i\frac{\pi}{4}}+e^ {i(\theta-\frac{\pi}{4})}\Big)
,\;\; u_{1}(0-)= \zeta (1+e^{i\theta}).
\end{align}
From the last equalities it follows that  
\begin{align}\label{Zcondi2}
\sum\limits_{j=2}^3 c_j^2u_j'(0+)-c_1^2u_1'(0-)=Zu_{1}(0-)
,\,\,
\text{where}\,\,
Z=Z(\theta)= -\sum\limits_{j=1}^3  c_j  \frac{e^{i\frac{\pi}{4}}+e^{i(\theta-\frac{\pi}{4})}}{1+e^{i\theta}}\in \mathbb R,
\end{align}
with $\theta\in [0, 2\pi)-\{\pi\}$. Thus, the set of self-adjoint extensions
$(\widehat{\mathcal M}, D(\widehat{\mathcal M}))$ of the symmetric operator $(\mathcal M, D(\mathcal M))$ can be seen as one-parametrized  family $(\mathcal J_Z, D(\mathcal J_Z))$ defined by \eqref{2trav23}. This finishes the proof.
\end{proof}

%
% \bibliography{referencias}
%
%%This was the original style for the bibliography
%\bibliographystyle{newstyle}
%

\def\cprime{$'\!\!$} \def\cprimel{$'\!$}

 \end{document}